\documentclass[11pt]{amsart}

\usepackage[margin=0.995in]{geometry}

\usepackage[alphabetic]{amsrefs}
\usepackage{bbm}
\usepackage{mathtools}
\usepackage{xypic}
\usepackage{microtype}

\usepackage{tikz}
\usepackage{tikz-cd}

\usepackage[all]{xy}
\usepackage{amsbsy}
\usepackage[bookmarksnumbered=true]{hyperref} 
\hypersetup{
     colorlinks = true,
     linkcolor = blue,
     anchorcolor = blue,
     citecolor = blue,
     filecolor = blue,
     urlcolor = blue
     }
\usepackage{xcolor}
\usepackage[shortlabels]{enumitem}
\usepackage{comment}

\colorlet{AW}{green!40!black!60!}
\colorlet{FG}{red!40!blue!60!}

\allowdisplaybreaks % so that align environment will break pages

\newcommand{\vect}[1]{\mathbf#1}
\newcommand{\R}{\mathbb R}
\newcommand{\Q}{\mathbb Q}
\newcommand{\Z}{\mathbb Z}

\newcommand{\kk}{\mathbbm{k}}  %needs package bbm

\newtheorem{theorem}{Theorem}[section]
\newtheorem{lemma}[theorem]{Lemma}

\newtheorem*{thm*}{Theorem}

\theoremstyle{definition}
\newtheorem{ex}[theorem]{Example}
\newtheorem*{def*}{Definition}
\newtheorem{dfn}[theorem]{Definition}
\newtheorem{rem}[theorem]{Remark}

\DeclareMathOperator{\charr}{char}
\DeclareMathOperator{\intt}{int}

\DeclareMathOperator{\supp}{supp}
\DeclareMathOperator{\trop}{trop}
\DeclareMathOperator{\Gr}{Gr}

\DeclareMathOperator{\cl}{cl}
\DeclareMathOperator{\val}{val}
\DeclareMathOperator{\Var}{V}
\DeclareMathOperator{\GL}{GL}

\newcommand{\KK}{\mathbbm{k}\{\!\{t\}\!\}}

\newcommand{\ourK}{K_{\mathbf{w}}}
\newcommand{\genPuiseux}{\kk\{\!\{x\}\!\} }
\newcommand{\bftheta}{\boldsymbol{\theta}}
\newcommand{\bfdelta}{\boldsymbol{\delta}}
\newcommand{\spec}{\phi^{\bftheta}_{\vect n}}
\newcommand{\bfgamma}{\boldsymbol{\gamma}}

\begin{document}

\title{Toric and tropical {B}ertini theorems in positive characteristic}

\author[Gandini]{Francesca Gandini}
%\address{Department of Mathematics, Kalamazoo College, 1200 Academy Street, Kalamazoo MI, 49006, USA }
%\email{francesca.gandini@kzoo.edu}

\author[Hering]{Milena Hering}
\address{}
\email{}

\author[Maclagan]{Diane Maclagan}
%\address{}
%\email{D.Maclagan@warwick.ac.uk}

\author[Mohammadi]{Fatemeh Mohammadi}
%\address{Department of Mathematics: Algebra and Geometry, Ghent University, 9000 Gent, Belgium\\ and Department of Mathematics and Statistics, UiT – The Arctic University of Norway, 9037 Troms\o, Norway}
%\email{fatemeh.mohammadi@ugent.be}

\author[Rajchgot]{Jenna Rajchgot}
%\address{Department of Mathematics and Statistics, McMaster University, 1280 Main Street West Hamilton, Ontario L8S 4K1, Canada}
%\email{rajchgot@math.mcmaster.ca}

\author[Wheeler]{Ashley K.\ Wheeler}
%\address{Department of Mathematics \& Statistics, Mount Holyoke College, South Hadley, MA, USA}
%\email{awheeler@mtholyoke.edu}

\author[Yu]{Josephine Yu}
% \address{School of Mathematics, Georgia Institute of Technology,     Atlanta GA, USA}
 %     \email {jyu@math.gatech.edu}

\begin{abstract}
  We generalize the toric Bertini theorem of Fuchs, Mantova, and
  Zannier \cite{FMZ} to positive characteristic.  A key part of the
  proof is a new algebraically closed field containing the field
  $\kk(t_1,\dots,t_d)$ of rational functions over an algebraically
  closed field $\kk$ of prime characteristic.  As a corollary, we
  extend the tropical Bertini theorem of Maclagan and Yu
  \cite{maclagan2019higher} to arbitrary characteristic, which
  removes the characteristic dependence from the $d$-connectivity
  result for  tropical varieties from that paper.
\end{abstract}

\maketitle

\section{Introduction}

Bertini's theorem, which states that a general hyperplane section of
an irreducible variety is again irreducible, is a basic result in
algebraic geometry.  This has been generalized in many different ways,
most notably for this paper by Fuchs, Mantova, and Zannier \cite{FMZ},
who replace hyperplane sections by certain subtori when the variety is
a subvariety of an algebraic torus in characteristic zero. Our main result
removes this characteristic assumption, at the expense of some precision.

\begin{theorem}[Toric Bertini] \label{t:maintoricbertini}
  Let $\kk$ be an algebraically closed field of arbitrary characteristic.
  Let $X$ be a $d$-dimensional irreducible subvariety of $(\kk^*)^n$
  with $d \geq 2$, and let $\pi \colon (\kk^*)^n \rightarrow (\kk^*)^d$ be a
  morphism with $\pi|_X$ dominant and finite.  Suppose that the pullback of $\pi|_X$  
  along any isogeny $\mu \colon (\kk^*)^d \rightarrow (\kk^*)^d$ is irreducible.
    Then for every $1 \leq r \leq d-1$ the set of
  $r$-dimensional subtori $T \subseteq (\kk^*)^d$ with $\pi^{-1}(\bftheta
  \cdot T) \cap X$ irreducible for all $\bftheta \in (\kk^*)^d$ is dense in the Grassmannian 
  $\Gr(r,d)$.
\end{theorem}
In characteristic zero \cite{FMZ} show that the conclusion holds for
 subtori $T$ in a generic (Zariski open) set, rather than just in a dense
 set.  More precisely, there are a finite number of exceptional
 subtori that $T$ must avoid, and any other $T$ will have irreducible
 preimage.  A similar ``generic'' conclusion may hold in
 characteristic $p$; we are not aware of a counterexample, though our
 current techniques are insufficient for a proof.  The techniques of
 \cite{FMZ} cannot be extended to characteristic $p$ as they rely on
 results that are false in characteristic $p$.

The condition on the irreducibility of pullbacks along isogenies appears already in
\cite{FMZ} and is necessary for a generic result;
for example, for the projection $\pi: \Var(x-y^2z^2)
\rightarrow (\kk^*)^2$ onto the first two coordinates, the preimage of any subtorus
of $(\kk^*)^2$ of the form $(t^{2a},t^b)$ is reducible.  Thus the set of desired subtori cannot be open.  However the denseness conclusion may still hold without this pullback hypothesis. We do not have a counterexample; see Remark~\ref{r:densenessnecessary}.

One consequence of Theorem~\ref{t:maintoricbertini} is that we can remove the
characteristic assumption in the Tropical Bertini theorem of Maclagan
and Yu \cite{maclagan2019higher}.

\begin{theorem}[Tropical Bertini] \label{t:tropicalBertini}
 Let $X \subset (\kk^*)^n$ be an irreducible $d$-dimensional variety,
 with $d \geq 2$, over an algebraically closed valued field $\kk$ with $\mathbb Q$
 contained in the value group.  The set of rational affine hyperplanes
 $H$ in $\mathbb R^n$ for which the intersection $\trop(X) \cap H$ is
 the tropicalization of an irreducible variety is dense in the
 Euclidean topology on $\mathbb P^n_{\mathbb Q}$.
\end{theorem}  

The characteristic zero case of the Tropical Bertini theorem was originally introduced in \cite{maclagan2019higher} to prove a
higher connectivity result for tropicalizations of irreducible
varieties \cite{maclagan2019higher}*{Theorem 1}.  Theorem~\ref{t:tropicalBertini} removes the
characteristic assumption from that connectivity theorem.

%introduce field

The key ingredient in the proof of Theorem~\ref{t:maintoricbertini} is a new field $\ourK$ that contains the algebraic
closure of $\kk(t_1,\dots,t_d)$. This field is smaller, in some crucial
aspects, than previously constructed 
algebraically closed fields containing $\kk(t_1,\dots,t_d)$ when $\kk$ has characteristic $p$.

The use for this field is best illustrated by the case when the
variety $X$ of Theorem~\ref{t:maintoricbertini} is a hypersurface, the
map $\pi \colon X \rightarrow (\kk^*)^d$ is projection onto the
first $d$ coordinates, and the subtorus $T \subseteq (\kk^*)^d$ is
one-dimensional: $T = \{ (t^{n_1},\dots,t^{n_d}) : t \in \kk^* \}$
for some $\mathbf{n} = (n_1,\dots,n_d) \in \mathbb Z^d$.  This is, in
fact, the core case of the proof.  Given an irreducible $f \in \kk[t_1^{\pm
    1},\dots,t_d^{\pm 1},y]$, we wish to show
that the set of $\mathbf{n} \in \mathbb Z^d$, for which the substitution $g \in
\kk[x^{\pm 1},y]$ given by
$$g(x,y) = f(x^{n_1},\dots,x^{n_d},y)$$ remains irreducible, is dense
in $\mathbb P^{d-1}_{\mathbb Q}$.  The key idea is to regard $f$ as a polynomial in $y$ with coefficients in $\kk(t_1,\dots,t_d)$.  The polynomial $f$ then factors as
$$f = \prod_{i=1}^s (y-\alpha_i),$$ where $\alpha_i$ are in the
algebraic closure of $\kk(t_1,\dots,t_d)$.  A precise description of
the algebraic closure is still unknown but several fields
containing it are known; see \cites{Christol,KedlayaEarly,KedlayaErratum,AdamczewskiBell,McD,GonzalezPerez, ArocaIlardi,ArocaRond, Saavedra}.  When $d=1$ and $\charr(\kk)=0$, the
algebraic closure of $\kk(t_1)$ is contained in the field of Puiseux
series $\kk\{\!\{t_1\}\!\} = \bigcup_{n \geq 1} \kk(\!(t_1^{1/n})\!)$.  The
exponents appearing in a particular Puiseux series all have a fixed
common denominator.  In characteristic $p$ this is relaxed to allow
arbitrary powers of $p$ in the denominator, subject to the
requirement, which is automatic for Puiseux series, that the set
of all exponents is well ordered.  As we recall in
Section~\ref{sec:fieldFamilies}, there are also multivariate
generalizations of this, so we may regard the roots $\alpha_i$ of $f$
as multivariate Laurent series with fractional exponents.

Given this description of the roots of $f$, the natural expectation is
that the roots of $g$ as a polynomial in $y$ are obtained from the
series $\alpha_i$ by the same specialization $t_i = x^{n_i}$.  To show
that $g$ remains irreducible for most $\mathbf{n}$ it then suffices to
show that none of these specializations of $\alpha_i$s, or the elementary symmetric
polynomials in them, which are the coefficients of any factors of $g$, are polynomials in $x$, as opposed to generalized Puiseux series.

However it is far from clear that this specialization map is well
defined.  For example, given the multivariate series
$$\alpha = \sum_{j \geq 1} t_1^{1-1/2^j}t_2^{1/2^j}$$ over a field of
characteristic $2$, we cannot make the substitution $t_1=t_2=x$.  A
contribution of this paper is to define a field $\ourK$
that contains the algebraic closure of $\kk(t_1,\dots,t_d)$, and has
natural subrings on which this specialization map is a well-defined
ring homomorphism, so this plan of attack goes through.

The structure of the paper is as follows.  The new field $\ourK$ is
introduced in Section~\ref{sec:fieldFamilies}, and it is shown to be
algebraically closed in Section~\ref{sec:algebraicallyClosed}.
Theorem~\ref{t:maintoricbertini} is proved in Section~\ref{sec:toricBertini}, while
Theorem~\ref{t:tropicalBertini} is proved in
Section~\ref{sec:tropicalBertiniTheorem}.

\noindent {\bf Acknowledgements.}  This project began at the BIRS
workshop on Women in Commutative Algebra, and we are grateful to the
organizers for bringing us together, and to BIRS for hosting.
Maclagan was partially supported by EPSRC grant EP/R02300X/1. Mohammadi was partially supported by EPSRC Early Career Fellowship EP/R023379/1, UGent BOF/STA/201909/038, and FWO grants (G023721N, G0F5921N). Rajchgot was partially supported by NSERC Grant RGPIN-2017-05732.  Yu was partially supported by NSF-DMS grant \#1855726.

\noindent {\bf Notation.}  Throughout this paper, by a cone in
$\mathbb R^d$ we mean a convex set closed under multiplication by
positive scalars.  It is {\em pointed} if its closure does not contain
a line.  For a cone $C$ in $\mathbb R^d$, the {\em dual cone}
$C^{\vee}$ is $\{ \mathbf{x} \in \mathbb R^d : \mathbf{x} \cdot
\mathbf{y} \geq 0 \text{ for all } \mathbf{y} \in C \}$.  We denote by
$\intt(C)$ the interior of a cone $C$, and by $\mathbb
R^d_{li}$ elements of $\mathbb R^d$ whose coordinates are linearly
independent over $\mathbb Q$.

\section{Field Families and $p$-discreteness}
\label{sec:fieldFamilies}

In this section we construct the field $\ourK$ containing
$\kk(t_1,\dots,t_d)$ that plays a key role in the proof of
Theorem~\ref{t:maintoricbertini}.  We will prove in
Section~\ref{sec:algebraicallyClosed} that $\ourK$ is algebraically closed.

\subsection{Field families}

We first recall previous constructions of algebraically closed fields
containing $\kk(t_1,\dots,t_d)$.

When $\charr(\kk)=0$ and $d=1$, the field of Puiseux series $\KK :=
\bigcup_{n \geq 0} \kk(\!(t^{1/n})\!)$ is algebraically closed and contains
$\kk(t)$.
Puiseux series are not algebraically closed when $\charr(\kk)=p>0$.
This was first observed by Chevalley \cite{Chevalley}, who observed that
the Artin--Schreier polynomial $x^p-x-t^{-1}$ has no Puiseux series root.
Abhyankar \cite{abhyankar} showed that the series $\sum_{j \geq 0}
t^{1/p^j}$ is a root.  This was generalized by Rayner
\cite{Rayner1}, who showed that the collection
\begin{equation}\label{eqtn:Rayneronevar}
  \left\{ \sum c_{\alpha} t^{\alpha} :
   \{ \alpha : c_{\alpha} \neq 0 \}\text{ is well ordered, and there is }
 N>0 \text{ with }  \alpha \in \bigcup_{j\geq 0}\frac{1}{Np^j}\Z  \text{ for all } \alpha \right\}
 \end{equation}
forms an algebraically closed field.
Rayner introduced the idea of a
  {\em field family}, which gives conditions on the possibilities of
  supports for power series of this form.  We now recall this in our setting.

\begin{dfn}
\label{dfn: fieldfamily}
Let $\mathcal A$ be a family of subsets of 
an ordered abelian group $(\Gamma, \preceq)$.
For $A\in\mathcal A$, let $S(A)$ denote the semigroup generated by $A$ under addition in $\Gamma$.
Then $\mathcal A$ is called a \emph{field family} with respect to
$\Gamma$ if the following axioms hold:
\begin{enumerate}[(i)]
\item Every $A\in\mathcal A$ is well ordered, i.e., every subset has a least element with respect to $\preceq$.
\item The elements of members of $\mathcal A$ generate $\Gamma$.
\item If $A,B\in\mathcal A$ then $A\cup B\in \mathcal A$.
\item If $A\in\mathcal A$ and $B\subseteq A$ then $B\in\mathcal A$.
\item Given $A\in\mathcal A$ and $\gamma\in\Gamma$, then $\gamma+A:=\{\gamma+a:\ a\in A \}\in \mathcal A$.
\item If $A\in\mathcal A$,
  with $a \succeq 0$ 
  for all $a \in A$, then $S(A)\in\mathcal A$.
\end{enumerate}
\end{dfn}

Given any field $\kk$, let $\kk^{\Gamma}$ denote the set of all mappings $\varphi:\Gamma\to \kk$.  The \emph{support} of $\varphi\in \kk^{\Gamma}$ is the set $\supp{\varphi} = \{\gamma\in\Gamma:\ \varphi(\gamma)\neq 0\}$.
For a family $\mathcal A$ of subsets of $\Gamma$ define
\[
\kk^{\Gamma}(\mathcal A) = \{\varphi\in \kk^{\Gamma}:\ \supp{\varphi}\in\mathcal A \}.
\]

We think of elements in $\kk^{\Gamma}(\mathcal A)$ as formal power series,
\[
\varphi = \sum_{\gamma\in \Gamma}\varphi(\gamma)\vect t^{\gamma} \in \kk^{\Gamma}(\mathcal A).
\]
For $\varphi, \psi\in \kk^{\Gamma}(\mathcal A)$ and $\gamma\in\Gamma$, addition is defined as
\[
(\varphi + \psi)(\gamma) = \varphi(\gamma) + \psi(\gamma)
\]
and multiplication is defined as
\[
(\varphi\cdot\psi)(\gamma) = \sum_{\gamma_1+\gamma_2=\gamma}\varphi(\gamma_1)\psi(\gamma_2).
\]

\begin{thm*}[\cite{Rayner1}*{Theorem 1}]
If $\mathcal A$ is a field family then $\kk^{\Gamma}(\mathcal A)$ is a field.
\end{thm*}

The field in \eqref{eqtn:Rayneronevar} is the case when $\Gamma = \mathbb Q$ and 
$\mathcal A$ consists of all well-ordered subsets of $\mathbb Q$ that lie in $\bigcup_{j \geq 0} \frac{1}{Np^j} \mathbb Z$ for some $N \in \mathbb N$.

When $\Gamma = \mathbb Q^d$ the field elements can be regarded as
generalized Puiseux series in $d$ variables $t_1,\dots,t_d$.
For $\charr(\kk)=0$, in~\cite{McD} McDonald constructed multivariate Puiseux series
solutions~$g$ for equations of the form $f(t_1,\dots,t_d,g)=0$ where $f
\in \kk[t_1,\dots,t_d,y]$.  The supports of the solution series are contained in sets
of the form $\frac{1}{N}
\mathbb Z^d \cap C$ where $N \in \mathbb N$ and $C$ is a pointed
cone.

Using Rayner's field family formalism
Aroca and Ilardi constructed algebraically closed fields
containing $\kk(t_1,\dots,t_d)$ and McDonald's series~\cite{ArocaIlardi}.
For $\charr(\kk)=p>0$ Saavedra gave an analogous construction of
algebraically closed fields containing $\kk(t_1,\dots,t_d)$
\cite{Saavedra}.  The supports of Saavedra's series are contained in sets
of the form $\bigcup_{j\geq 0}\frac{1}{Np^j}
\mathbb Z^d \cap C$ where $N \in \mathbb N$ and $C$ is a pointed
cone.

All of these fields contain elements that are not algebraic over
$\kk(t_1,\dots,t_d)$.  While there is some work on characterizing the
algebraic closure of $\kk(t_1,\dots,t_d)$
\cites{AdamczewskiBell,ArocaRond,Christol,KedlayaEarly,KedlayaErratum}
there is currently no complete answer in multiple variables and
arbitrary characteristic.

\subsection{The field of $p$-discrete series}
\label{sec: p-discreteness}

We now use the formalism of Definition~\ref{dfn: fieldfamily} to
construct a field containing $\kk(t_1,\dots,t_d)$ but contained in
Saavedra's field.  Let $\R^d_{li}$ denote the set of vectors in $\R^d$
whose coordinates are linearly independent over $\Q$.  For $\vect w\in\R^d_{li}$
we define an ordering $\preceq_{\vect w}$ on the abelian group $\Q^d$
as follows:
\[
\vect u\preceq_{\vect w} \vect v\quad\text{ when }\quad \vect w\cdot \vect u\leq \vect w\cdot \vect v.
\]
The condition that the entries of $\vect w$ are linearly independent over the rationals ensures that $\preceq_{\vect w}$ is a total order.

\begin{dfn}[$p$-discreteness]
\label{dfn: p-discrete}
For a subset $A \subseteq \mathbb Q^d$, $\bfgamma \in \mathbb Q^d$, 
and
$\mathbf{w} \in \mathbb R^d_{li}$, set 
\[
A^+_{\bfgamma, \mathbf{w}} = \{\vect a
\in A : \mathbf{w} \cdot \vect a > \mathbf{w} \cdot \bfgamma\}\ \ \text{and}\ \ 
A^-_{\bfgamma, \mathbf{w}} = \{\vect a \in A : \mathbf{w} \cdot \vect a <
\mathbf{w} \cdot \bfgamma \}.
\]
We say that a set $A \subseteq \mathbb Q^d$ is \emph{$p$-discrete with
  respect to $\mathbf{w}$} if the following conditions are satisfied.
\begin{enumerate}[(a)]
\item 
\label{item:perturbation} 
There exists an open cone  $\sigma$ containing $\mathbf{w}$ for which the set $\{\mathbf{w}'\cdot \mathbf{a}  : \mathbf{a} \in A \}$ is well ordered when $\mathbf{w}' \in \sigma \cap \mathbb R^d_{li}$.
\item 
\label{item:cone} 
There is $N>0$, $\bfgamma \in \mathbb Q^d$, a pointed  cone $C$ with $\mathbf{w} \in \intt(C^{\vee})$ such that
\[
A \subseteq (\bfgamma+C) \cap \left(\bigcup_{j \geq 0} \frac{1}{Np^j} \mathbb Z^d\right).
\]
\item 
\label{item:convergence} 
For any sequence $\{\vect a_i \}$ in $A$, if the sequence $\{
   \mathbf{w} \cdot \mathbf{a}_i \}$ converges in $\mathbb R$ then
  $\{\vect a_{i}\}$ converges to a point of $\mathbb Q^d$.
\item 
\label{item:plusminus} 
For all $\bfgamma' \in \mathbb Q^d$ there
  is an open cone $\sigma_{\bfgamma'}$ containing $\mathbf{w}$ for which for
  all $\mathbf{w}' \in \sigma_{\bfgamma'} \cap \mathbb R^d_{li}$, we have
  $A^+_{\bfgamma',\mathbf{w}'} =A^+_{\bfgamma',\mathbf{w}}$ and
  $A^-_{\bfgamma',\mathbf{w}'} =A^-_{\bfgamma',\mathbf{w}}$.

\end{enumerate}
\end{dfn}

For a given $p$-discrete set $A$ we may choose the closure of $\sigma$ in
condition~\ref{item:perturbation} and the dual cone $C^{\vee}$ of $C$ in
condition~\ref{item:cone} to coincide, at the expense of at least one
of them not being the largest possible, since $\sigma$ can be replaced
with a smaller open cone and $C$ can be replaced with a larger pointed
cone.  In condition~\ref{item:plusminus} it is
possible that a different $\sigma_{\bfgamma'}$ is needed for each
$\bfgamma'$.  Note that condition~\ref{item:plusminus} is equivalent
to the existence of open cones $\sigma_{\bfgamma'}$ containing $\mathbf{w}$ such that  that
$A^-_{\bfgamma',\mathbf{w}} \subseteq \bfgamma' -
\sigma_{\bfgamma'}^{\vee}$, and $A^+_{\bfgamma',\mathbf{w}} \subseteq
\bfgamma'+\sigma_{\bfgamma'}^{\vee}$.

\begin{dfn}
For fixed $\mathbf{w} \in \mathbb R^{d}_{li}$,  we define $\mathcal A_{\mathbf{w}}$ to be the
set
\begin{equation}
\label{eqtn:charpfieldfamily}
  \mathcal A_{\mathbf{w}} = \{ A \subset \mathbb Q^d:
  A \text{ is } p \text{-discrete with respect to }\mathbf{w} \}.\tag{$*$}
\end{equation}
\end{dfn}
\begin{rem}
  Saavedra's field family \cite{Saavedra}*{Proposition 5.1} satisfies
  condition~\ref{item:cone}, and also a weaker version of
  condition~\ref{item:perturbation}, where the well ordering is
  required only for $\mathbf{w}$.  The stronger condition~\ref{item:perturbation} is needed in
  Section~\ref{sec:toricBertini} to show that some specialization maps
  $t_i = x^{n_i}$ are well defined.   Condition~\ref{item:plusminus}
  is needed in the proof of algebraic closure in
  Section~\ref{sec:algebraicallyClosed} and is also used in Section~\ref{sec:toricBertini}.
  Condition~\ref{item:convergence} is
  needed in the proof of Lemma~\ref{condition d} to show that these
  series form a field family. 
\end{rem}

In the rest of this section we show that  $\mathcal A_{\mathbf{w}}$ is a field family with respect to $(\Q^d,\preceq_{\mathbf{w}})$.

\begin{lemma} 
\label{lem:bounded}
Fix $\vect w\in \R^d_{li}$, and $\bfgamma\in\Q^d$.  Suppose $A\in \mathcal A_{\vect w}$ with $\mathbf{w} \cdot \vect a \geq 0$ for all $\vect a \in A$.
Let $S(A)$ be the semigroup generated by $A$ under addition in $\Q^d$ and let $S(A)^-_{\bfgamma,\vect w}$ be as in Definition~\ref{dfn: p-discrete}.  Then there is $M>0$ such that for any nonzero  $\vect s=\sum_{i=1}^m\vect a_i\in S(A)^-_{\bfgamma,\vect w}$, where $\vect a_i \in A \setminus \{0\}$, the number $m$ of summands is at most $M$.%} 
\end{lemma}

\begin{proof}
  We have $\mathbf{w}\cdot \vect s = \sum_{i=1}^m \mathbf{w} \cdot \vect a_i$, so 
\[0< \min_i \{ \mathbf{w}
\cdot \vect a_i \} \leq \frac{\mathbf{w} \cdot \vect s}{m} < \frac{\mathbf{w} \cdot \bfgamma}{m},
\]
where the last inequality holds because $\vect s\in
S(A)^-_{\bfgamma,\vect w}$.  If for all $m > 0$ there is
$\mathbf{s} \in S(A)^-_{\bfgamma, \mathbf{w}}$ that is the sum of at least $m$ nonzero terms, we would
get a contradiction to well-ordering of the set $\{\vect w\cdot \vect
a \colon \vect a \in A\}$ given in condition~\ref{item:perturbation}.
\end{proof}

\begin{lemma}\label{condition d}
  Let $A\in\mathcal A_{\vect w}$ and let $S(A)$ be the semigroup generated by elements of $A$ under addition.  Suppose that $\mathbf{w} \cdot \mathbf{a} \geq 0$ for all $\vect a \in A$, and  
  conditions~\ref{item:perturbation}, \ref{item:cone},
  \ref{item:convergence} of $p$-discreteness hold for $S(A)$.
  Then $S(A)$ also satisfies
  condition~\ref{item:plusminus}, so $S(A)$ is in $\mathcal
  A_{\mathbf{w}}$.
\end{lemma}

\begin{proof}
 Fix $\bfgamma' \in \mathbb Q^d$.  It suffices to show there are open cones 
$\sigma_1$ and $\sigma_2$ containing $\vect w$ for which
$S(A)^+_{\bfgamma',\vect w}
\subseteq S(A)^+_{\bfgamma',\vect w'}$  for all 
$\vect w'\in \sigma_1\cap\R^d_{li}$  and 
$S(A)^-_{\bfgamma',\vect w} \subseteq S(A)^-_{\bfgamma',\vect w'}$   
for all
$\vect w'\in \sigma_2\cap \R^d_{li}$. 

We first show the existence of $\sigma_1$.  Since $\{ \mathbf{w} \cdot \vect a :\vect a\in S(A)^+_{\bfgamma', \vect w} \}$ 
 is
well ordered by \ref{item:perturbation} there is $\vect a \in
S(A)^+_{\bfgamma',\mathbf{w}}$ achieving the minimum value of
$\mathbf{w} \cdot \vect a$, so 
\[
\delta:= \min\{\mathbf{w} \cdot \vect a : \vect a \in
S(A)^+_{\bfgamma',\mathbf{w}}\} - \mathbf{w} \cdot \bfgamma'
\]
is positive.  By condition~\ref{item:cone}, we have $S(A) \subseteq
\bfgamma +C$ for a pointed cone $C$ with $\mathbf{w} \in
\intt(C^{\vee})$ and $\bfgamma \in \mathbb Q^d$.  If $\mathbf{w} \cdot
\bfgamma' \leq \mathbf{w} \cdot \bfgamma$ then we may take $\sigma_1$
to be the interior of $\{ \mathbf{x} : \mathbf{x} \cdot \bfgamma
\geq \mathbf{w}' \cdot \bfgamma' \} \cap C^{\vee}$.  We now suppose
that $\mathbf{w} \cdot \bfgamma' > \mathbf{w} \cdot \bfgamma$.  We may
assume that $C$ is full-dimensional here, so there is a rational point
$\tilde{\bfgamma}$ in $\bfgamma + C \cap \{ \mathbf{x} \in \mathbb R^d
\colon \mathbf{w} \cdot \bfgamma' <\mathbf{w} \cdot \mathbf{x} <
\mathbf{w} \cdot \bfgamma' + \delta\}$.  Let $C'$ be the closure of
the cone generated by $\{ \vect a-\tilde{\bfgamma} : \vect a \in
S(A)^+_{\bfgamma',\mathbf{w}} \}$.  The cone $C'$ is pointed, as it is
contained in the cone with vertex $\tilde{\bfgamma}$ over the
intersection $\bfgamma+ C \cap \{\vect x:\mathbf{w} \cdot \vect x =
\mathbf{w} \cdot \bfgamma' + \delta\}$, which is bounded.  It also
contains $\mathbf{w}$ in its dual.  Thus any open cone $\sigma_1$
containing $\mathbf{w}$ and contained in $\intt({C'}^{\vee}) \cap
\{\mathbf{x} \in \mathbb R^d \colon \vect x \cdot (\tilde{\bfgamma} -
\bfgamma') \geq 0 \}$ will have the property that
$S(A)^+_{\bfgamma',\vect w} \subseteq S(A)^+_{\bfgamma',\vect w'}$ for
all $\vect w'\in \sigma_1$.

We next show the existence of $\sigma_2$.

By Lemma~\ref{lem:bounded} there is a bound $M$ on the number of summands $m$ in any
element $\vect a \in S(A)^-_{\bfgamma',\mathbf{w}}$.
Let
\[
B_m = \left\{ (\vect{a_1},\dots,\vect a_m) \in A^m : \sum_{i=1}^m \vect a_i \in
S(A)^-_{\bfgamma',\mathbf{w}}\right\}.
\]
We show by induction on $m$ that there
is a cone $C_m$ with $\mathbf{w} \in \intt(C_m^{\vee})$ and
$\{ \sum_{i=1}^m \vect a_i \in S(A)^-_{\bfgamma',\vect w} \} \subseteq \bfgamma'-C_m$.  Any
$\sigma_2 \subseteq \intt(\bigcap_{m=1}^M C_m^{\vee})$ 
containing $\mathbf{w}$
will then have the property that $S(A)^-_{\bfgamma',\vect w} \subseteq
S(A)^-_{\bfgamma',\vect w'}$ for all $\vect w'\in \sigma_2$.

The base case is $m=1$, where the claim follows since $B_1 \subseteq
A$, so $B_1 \in \mathcal A_{\mathbf{w}}$, for which axiom~\ref{item:plusminus}
holds. 
Now suppose that $m>1$, and the claim is true for smaller $m$.
If 
\[
s:=\sup\left\{\mathbf{w} \cdot \sum_{i=1}^m \vect a_i : (\vect a_1,\ldots, \vect a_m) \in B_m\right\}
\]
is less than $\mathbf{w} \cdot \bfgamma'$,
then the closure $C_m$ of the cone generated by $
\{ \bfgamma'- \sum_{i=1}^m \vect a_i : (\vect a_1,\ldots, \vect a_m) \in
B_m \}$ is pointed and has the required form.   We may thus assume that
$s= \mathbf{w} \cdot \bfgamma'$.

For each $n>0$, the set
\[
\mathcal S^{(0)}_n= \left\{ \vect a_1 : (\vect a_1,\ldots, \vect a_m)\in B_m, \mathbf{w} \cdot \sum_{i=1}^m \vect a_i
> \mathbf{w} \cdot \bfgamma'-1/n \right\}
\]
is a subset of $A$, so by axiom
\ref{item:perturbation} there is some  $\vect s_n^{(0)} \in \mathcal S_n^{(0)}$ with $\mathbf{w} \cdot \vect s_n^{(0)}$
minimal.
We claim that the set $\{\mathbf{w} \cdot \vect s_n^{(0)} : n \geq 1\}$ is
weakly increasing, and bounded above by $\mathbf{w} \cdot
\bfgamma'$. The bound comes from the fact that $\mathbf{w} \cdot \vect
a_i \geq 0$ for all $\vect a_i$ by assumption, which implies that
$\mathbf{w} \cdot \vect s_n^{(0)} = \mathbf{w} \cdot \vect a_1 \leq
\mathbf{w} \cdot \sum_{i=1}^m \vect a_i < \mathbf{w} \cdot
\bfgamma'$. Thus the sequence $\{ \mathbf{w} \cdot \vect s_n^{(0)} \}$
converges to $\ell_0 \in \mathbb R$.
We now iterate, at each stage constructing sets
\[
\mathcal T_{j} = \{ (\vect a_1,\ldots, \vect a_m)\in  B_m  : \mathbf{w} \cdot \vect a_1> \ell_{j-1}\},
\]
and
\[
\mathcal S_n^{(j)} = \left\{ \vect a_1 : (\vect a_1,\ldots, \vect a_m)
\in \mathcal T_j, \mathbf{w} \cdot \sum_{i=1}^m \vect a_i > \mathbf{w} \cdot
\bfgamma'-1/n \right\}.
\]
with sequences $\{\vect s_n^{(j)}\}$, for which $\mathbf{w} \cdot
\mathbf{s}_n^{(j)}$ converges to $\ell_j$.  Note that $\ell_j > \ell_{j-1}$ by construction.

We claim that this process must terminate for some $j$, with $\sup\{
\mathbf{w} \cdot \sum_{i=1}^m \vect a_i : (\vect a_1,\ldots, \vect a_m) \in \mathcal T_j\} <
\mathbf{w} \cdot \bfgamma'$.  To see this, suppose that we can construct
an infinite sequence
\[ \ell_0 < \ell_1 < \ell_2 < \ell_3 < \dots \]
For each $j>1$, fix $0<\epsilon_j < \min\{( \ell_{j+1} - \ell_j)/2, 
(\ell_j-\ell_{j-1})/2 \}$.  Since for each $j$, the sequence
$\{\mathbf{w} \cdot \vect s_n^{(j)}\}$ converges to $\ell_j$ there is
$N_j>0$ for which $|\mathbf{w} \cdot \vect s_n^{(j)} - \ell_j|<
\epsilon_j$ for $n>N_j$.  Fix $n_j>N_j$ with $1/n_j<\epsilon_j$ and
pick $\vect c_j=(\vect a_1,\ldots, \vect a_m) \in
\mathcal T_j$
with $\vect a_1
=\vect s_{n_j}^{(j)}$.  Then $\mathbf{w} \cdot \sum_{i=1}^m \vect a_i >
\mathbf{w} \cdot \bfgamma' - \epsilon_j$.  Because $\sum_{i=1}^m \vect
a_i \in S(A)^-_{\bfgamma', \vect w}$, we also have the condition that
$\vect w\cdot \sum_{i=1}^m\vect a_i<\vect w\cdot \bfgamma'$. Subtracting
$\vect w\cdot \vect s_{n_j}^{(j)} = \vect w \cdot \vect a_1$ we get
\[
\mathbf{w} \cdot \bfgamma' - \epsilon_j - \mathbf{w} \cdot \vect
s_{n_j}^{(j)} < \sum_{i=2}^m \mathbf{w} \cdot \vect a_i < \mathbf{w}
\cdot \bfgamma' - \mathbf{w} \cdot \vect s_{n_j}^{(j)}.
\]
Hence, since $-\ell_j + \epsilon_j > -\mathbf{w} \cdot  \vect s_{n_j}^{(j)} \geq -\ell_j$ and by the choice of $\epsilon_j$, we have 
\[
\mathbf{w} \cdot \vect \bfgamma' - \ell_{j+1} < \mathbf{w}
\cdot \vect \bfgamma' - \ell_j -  \epsilon_j < \sum_{i=2}^m
\mathbf{w} \cdot \vect a_i < \mathbf{w} \cdot \bfgamma' - \ell_j+\epsilon_j< \mathbf{w} \cdot \bfgamma' - \ell_{j-1}.
\]
But this implies that the subset $\{ \sum_{i=2}^m \vect a_i : \vect
c_j=(\vect a_1,\ldots, \vect a_m) \}$ of $S(A)$ is not well ordered,
contradicting our assumption that condition \ref{item:perturbation}
holds. From this contradiction we conclude that the process
terminates, so there is $j$ for which $\sup\{ \mathbf{w} \cdot
\sum_{i=1}^m \vect a_i : (\vect a_1,\ldots, \vect a_m) \in \mathcal
T_j\} < \mathbf{w} \cdot \bfgamma'$.  Write $L$ for this $j$ and set
\[\epsilon:=\mathbf{w} \cdot \bfgamma' - \sup\left\{\mathbf{w}
\cdot \sum_{i=1}^m \vect a_i : (\vect a_1,\ldots, \vect a_m) \in \mathcal T_L\right\},
\]
which is positive.

For each $0 \leq j \leq L-1$, set 
\[
\mathcal T_j' = \left\{  (\vect a_1,\ldots, \vect a_m) \in B_m: \mathbf{w}
\cdot \sum_{i=2}^m \vect a_i > \mathbf{w} \cdot \bfgamma' - \ell_{j} \right\}.
\]
The set $\{ \mathbf{w} \cdot \sum_{i=2}^m \vect a_i : (\vect a_1,\ldots, \vect a_m) \in \mathcal T_j' \}$ is well ordered by \ref{item:perturbation}, so 
\[
\mu_j :=
\min\left\{\mathbf{w}~\cdot~\sum_{i=2}^m \vect a_i - \mathbf{w} \cdot \bfgamma'
+\ell_{j} : (\vect a_1,\ldots, \vect a_m) \in \mathcal T'_j\right\}
\] 
exists and is positive.  Choose $n \in \mathbb N$ with
$1/n<\min_j(\mu_j,\epsilon)$, and $n'>n$ with $\ell_{j}-\mathbf{w} \cdot \vect s_{n'}^{(j)}<1/n$ for all $0 \leq j <L$.

Set $\mathcal T_0 = B_m$,
and fix $0 \leq j < L$.  Since $\ell_{j}$ is the limit of the sequence $\{
\mathbf{w} \cdot \mathbf{s}_n^{(j)} \}$, for each $j$, by condition
\ref{item:convergence} we have $\ell_{j} = \mathbf{w} \cdot
\tilde{\bfgamma}_j$ for $\tilde{\bfgamma}_j \in \mathbb Q^d$.  We
first consider the case that $(\vect a_1,\dots,\vect a_m) \in \mathcal
T_j \setminus \mathcal T_{j+1}$ satisfies
\begin{equation}\label{eq:*}
  \vect w \cdot \sum_{i=1}^m \vect a_i > \vect w \cdot \bfgamma' - 1/n'.
\end{equation}
In that case $\mathbf{w} \cdot \vect a_1 \geq \mathbf{w} \cdot \vect
s_{n'}^{(j)} > \ell_{j} - 1/n$.  Thus $\mathbf{w} \cdot \sum_{i=2}^m
\vect a_i < \mathbf{w} \cdot \bfgamma' - \ell_{j}+1/n < \mathbf{w}
\cdot \bfgamma'-\ell_{j}+\mu_j$, so by the definition of $\mu_j$ we have
$\mathbf{w} \cdot \sum_{i=2}^m \vect a_i \leq \mathbf{w} \cdot \bfgamma'
- \ell_{j}$, and thus $\mathbf{w} \cdot \sum_{i=2}^m \vect a_i \leq
\mathbf{w} \cdot \bfgamma' - \mathbf{w} \cdot \tilde{\bfgamma}_j$.

By induction there is a cone $C_{j,m-1}$ with $\mathbf{w} \in
\intt(C_{j,m-1}^{\vee})$ and
\[
\left\{\sum_{i=2}^m \vect a_i : (\vect a_1,\ldots, \vect a_m) \in B_m,
\mathbf{w} \cdot \sum_{i=2}^m \vect a_i \leq \mathbf{w} \cdot
\bfgamma'-\mathbf{w} \cdot \tilde{\bfgamma}_j \right\} \subseteq \bfgamma'-
\tilde{\bfgamma}_j -C_{j,m-1}.
\] 
By \ref{item:plusminus} for $A$ there is an open cone $C'_j$ with
$\mathbf{w} \in \intt({C'_j}^{\vee})$ and $\{\vect a_1 : \mathbf{w}
\cdot \vect a_1 \leq \ell_{j} = \mathbf{w} \cdot \tilde{\bfgamma}_j \}
\subseteq \tilde{\bfgamma}_j -C'_j$.  Thus if $(\vect a_1,\dots, \vect
a_m) \in \mathcal T_j \setminus \mathcal T_{j+1}$ with $\vect w \cdot
\sum_{i=1}^m \vect a_i > \vect w \cdot \bfgamma' - 1/n'$
then
$$ \sum_{i=1}^m \vect a_i \in \bfgamma' - (C_j'+C_{j,m-1}).$$ Let
$C''$ be the Minkowski sum $\sum_{j=0}^{L-1} (C_j'+C_{j,m-1})$.  Note
that for $(\vect a_1,\dots,\vect a_m) \in \mathcal T_L$ we have
$\mathbf{w} \cdot \sum_{i=1}^m \vect a_i < \mathbf{w} \cdot \bfgamma'
- 1/n'$.  Thus for $(\vect a_1,\dots,\vect a_m) \in B_m$ with
$\mathbf{w} \cdot \sum_{i=1}^m \vect a_i > \mathbf{w} \cdot \bfgamma'
- 1/n'$ we have
$$ \sum_{i=1}^m \vect a_i \in \bfgamma'  - C''.$$
The closure of the cone $C'''$ generated by
\[
\left\{ \bfgamma' - \sum_{i=1}^m \vect a_i : (\vect a_1,\ldots, \vect a_m) \in B_m, \mathbf{w} \cdot
\sum_{i=1}^m \vect a_i \leq \mathbf{w} \cdot \bfgamma'-1/n' \right\}
\] 
is pointed,
$\mathbf{w} \in \intt({C'''}^{\vee})$, and $\{ \sum_{i=1}^m \vect a_i : (\vect a_1,\ldots, \vect a_m)
\in B_m, \mathbf{w} \cdot \sum_{i=1}^m \vect a_i \leq \mathbf{w} \cdot
\bfgamma'-1/n' \} \subseteq \bfgamma' - C'''$.  

Finally, let $C_m = C'' +C'''$.  Then $\mathbf{w} \in
\intt(C_m^{\vee})$, and for all $(\vect a_1,\ldots, \vect a_m) \in
B_m$ with $\mathbf{w} \cdot \sum_{i=1}^m \vect a_i < \mathbf{w} \cdot
\bfgamma'$, we have $\sum_{i=1}^m \vect a_i \in \bfgamma' - C_m$ as
required.
\end{proof}

\begin{theorem}
\label{thm:fieldfamily}
The set $\mathcal A_{\mathbf{w}}$ 
is
a field family with respect to $(\mathbb Q^d,\preceq_{\mathbf{w}})$. 
\end{theorem}  

\begin{proof}
  We check each of the six axioms of a field family given in Definition~\ref{dfn: fieldfamily}.

\emph{Axiom (i)}:
  Each $A \in \mathcal A_{\mathbf{w}}$ is well ordered by condition~\ref{item:perturbation} of Definition~\ref{dfn: p-discrete}.

\emph{Axiom (ii)}:
  Every $\bfgamma \in \mathbb Q^d$ is in $\mathcal A_{\mathbf{w}}$.

\emph{Axiom (iii)}:
  Fix $A, B \in \mathcal A_{\mathbf{w}}$.  Condition~\ref{item:cone}
  of $p$-discreteness holds for $A \cup B$ by
  \cite{Saavedra}*{Proposition 5.1}.
  Let $\sigma_A, \sigma_B$ be the respective open cones containing $\vect w$ guaranteed by condition~\ref{item:perturbation} of $p$-discreteness for $A$ and $B$.  The open cone $\sigma = \sigma_A
  \cap \sigma_B$ makes condition~\ref{item:perturbation} hold for $A\cup B$.  Similarly, for condition~\ref{item:plusminus} we may use the intersection of the respective guaranteed open cones for $A$ and $B$.  

For \ref{item:convergence}, suppose $\{\vect c_i\}$ is a sequence in $A \cup B$
  with $\{\mathbf{w} \cdot \vect c_i\}$ converging to $L \in \R$.  Let $\{\vect a_j\}$ and $\{\vect b_j\}$
  be the subsequences consisting of points in $A$ and in $B$, respectively.  If one of these subsequences is finite, then $\{\vect c_i\}$ converges
  to the limit of the other sequence.   Otherwise, $\{\mathbf{w}
  \cdot \vect a_j\}$ and $\{\mathbf{w} \cdot \vect b_j\}$ also converge to $L$, so $\{\vect a_j\}$
  converges to $\vect a$ and $\{\vect b_j\}$ converges to $\vect b$ with $\vect a,\vect b \in \mathbb
  Q^d$.  We have $\mathbf{w} \cdot \vect a
  = \mathbf{w} \cdot \vect b = L$, so $\mathbf{w} \cdot (\vect a-\vect b) = 0$.  Since
  $\mathbf{w} \in \mathbb R^d_{li}$, this implies that $\vect a-\vect b = 0$, so
  $\vect a=\vect b$, and $\{\vect c_i\}$ converges to the point $\vect a=\vect b$ in $\mathbb Q^d$ as required.

\emph{Axiom (iv)}:
 All conditions of $p$-discreteness are inherited from $A$ by $B\subseteq A$. 

 \emph{Axiom (v)}: All conditions for a set $A\subseteq \Q^d$ to be $p$-discrete are invariant under translation by a point in $\mathbb Q^d$.

\emph{Axiom (vi)}:
  Fix $A \in \mathcal A_{\mathbf{w}}$ with $\mathbf{w} \cdot \vect a
  \geq 0$ for all $\vect a \in A$.  Let $S(A)$ be the semigroup
  generated by $A$ under addition.  Let $\sigma$ be as in
  condition~\ref{item:perturbation} of $p$-discreteness for $A$, and let $\sigma_{\mathbf{0}}$ be as in condition~\ref{item:plusminus} for $\mathbf{0}$.  For
  any $\mathbf{w}' \in \sigma \cap \sigma_{\mathbf{0}}$ we have $A^+_{0,\mathbf{w}'} = A^+_{0,\vect
    w} = A$, since $\mathbf{w} \cdot \mathbf{a} \geq 0$ for all
  $\mathbf{a} \in A$,
  and $A$ is $p$-discrete.  Thus the fact that $S(A)$ is
  well ordered with respect to $\mathbf{w}'$ follows from 
\cite{Neumann}*{Theorem 3.4}. Condition~\ref{item:cone} for $S(A)$ follows as in \cite{Saavedra}*{Proposition 5.1}. 

We now prove condition~\ref{item:convergence}  holds. 
Let $\{\vect s_i\}$ be a sequence in $S(A)$ such that $\{\vect w \cdot \vect s_i\}$ converges to some $L \in \mathbb{R}$. 
We want to show that $\{\vect s_i\}$  converges to some $\vect s \in \mathbb{Q}^d$. 
Pick $\bfgamma \in \mathbb Q^d$ such that $\vect w \cdot \bfgamma > L$.
There exists an integer $N$ such that for all $i>N$, we have $|\vect w\cdot \vect s_{i}-L|<\vect w\cdot \bfgamma-L$. Consequently, for all $i>N$ we have $\vect s_{i}\in S(A)^-_{\bfgamma, \vect w}$. 
When $\vect s_i\neq \vect 0$, Lemma~\ref{lem:bounded} says that the number $m$ of summands of  $\vect s_{i} = \sum_{j=1}^m \vect a^{(j)}_i$ is bounded by some $M>0$. When $\vect s_i = \vect 0$, the assumption that $ \vect w\cdot \vect a\geq 0$ for all $\vect a\in A$ implies that $\vect s_i =  \sum_{j=1}^m \vect a^{(j)}_i$ has just one summand, namely $\vect 0\in A$. 
Thus by passing to a subsequence $\{\vect s_{n(i)} \}$ of $\{\vect s_{i}\}$ 
we may assume that each $\vect s_{n(i)}$ is a sum of exactly $m$ terms $\vect s_{n(i)} =
\sum_{j=1}^m \vect a_{i}^{(j)}$ for some $m \leq M$, and the sequence
$\{\mathbf{w} \cdot \vect a_{i}^{(j)}\}_i$ is weakly increasing for each $1 \leq j \leq
m$.  By hypothesis we have 
$0 \leq \mathbf{w} \cdot \vect a_{i}^{(j)} < \mathbf{w} \cdot \bfgamma$. 
Hence for each fixed
$1 \leq j \leq m$ the weakly increasing and bounded sequence $\{\mathbf{w} \cdot \vect a_{i}^{(j)}\}_i$ converges. 
Since $A$ is $p$-discrete, each $\{\vect a_{i}^{(j)}\}_i$ converges to some $\vect a^{(j)}\in \Q^d$, so $\{\vect s_{n(i)}\}$ 
converges to $\vect s= \sum_{j=1}^m \vect a^{(j)} \in \mathbb Q^d$.

So far we have shown that $\{\vect s_i\}$ has a subsequence converging to $\vect s\in \mathbb{Q}^d$. We claim that the entire sequence $\{\vect s_i\}$ also converges to $\vect s$.  Suppose there exists an $\epsilon>0$ for which for all $n\in \mathbb{N}$, there is an index $i_n>n$ such that $|\vect s_{i_n}-\vect s|>\epsilon$. 
Note that $\{\mathbf{w}\cdot \vect s_{i_n}\}_n$  also converges to $L$. Hence, by applying the same argument as above, we deduce that the sequence $\{\vect s_{i_n}\}_n$ itself has a convergent subsequence to $\vect s$, contradicting the $\epsilon$ distance of the $\vect s_{i_n}$ away from $\vect s$.  

Condition~\ref{item:plusminus} of $p$-discreteness now holds by Lemma~\ref{condition d}, which completes the proof. \qedhere
\end{proof}

\begin{dfn}\label{dfn:K_w}
Let $\kk$ be an algebraically closed field of characteristic $p>0$ and
let $\vect w\in\R^d_{li}$. We denote by $K_{\vect w}$ the field
$\kk^{\Q^d}(\mathcal A_{\vect w})$ of $p$-discrete series in direction
$\mathbf{w}$ with coefficients in $\kk$, with variables
$t_1,\dots,t_d$.  Specifically, the elements of $K_{\vect w}$ are
formal power series in variables $t_1,\dots t_d$, with exponents in
the field family $\mathcal A_{\vect w}$ and coefficients in $\kk$.
\end{dfn}

\section{The field of p-discrete series is algebraically closed}
\label{sec:algebraicallyClosed}

In this section, we prove that $\ourK$  is algebraically closed. We begin with some preliminary results.

\begin{lemma}\label{lem:algClosedbounded}
  Let $A\in \mathcal{A}_{\bf w}$, and assume that $\vect w\cdot \vect a <0$ for all $\vect a\in A$. Then there is an $M>0$ such that $|\vect a|
  \leq M$ for all $\vect a\in A$. That is, $A$ is bounded.
\end{lemma}

\begin{proof}
Since $A \in \mathcal A_{\bf w}$, the set $A$ is contained in the
polyhedron $\bfgamma + C$ for some $\bfgamma \in \mathbb Q^d$ and  pointed
 cone $C$ with $\mathbf{w} \in \intt(C^{\vee})$.  The polyhedron $(\bfgamma + C) \cap \{ \mathbf{x} : \mathbf{w} \cdot \mathbf{x} \leq 0 \}$ is a polytope, as its facet normals span $\mathbb R^d$, so it is bounded, and thus $A$ is also bounded.
\end{proof}

\begin{lemma}\label{prop:algClosedpDiscrete}
Let $A\in \mathcal{A}_{\bf w}$, and assume that $\vect w \cdot \vect a<0$ for all $\vect a\in A$. Then $S := \bigcup_{i=1}^{\infty} p^{-i}A$ is also in the field family  $\mathcal{A}_{\bf w}$. 
\end{lemma}

\begin{proof}
 We check that $S$ satisfies the conditions of Definition~\ref{dfn: p-discrete}.
  
\noindent \emph{Condition~\ref{item:perturbation}:} We use similar ideas to \cite[Lemma 5.2]{Saavedra}. 

Since condition ~\ref{item:perturbation} holds for $A$, there exists an open  cone $\sigma$ containing $\mathbf{w}$ for which $Q_{i, \vect w'}:= \{\mathbf{w}' \cdot(p^{-i}\mathbf{a}): p^{-i}\mathbf{a} \in p^{-i}A \}$ is well ordered for any $i\in \mathbb{Z}_{\geq 0}$ and ${\bf w}'\in \sigma\cap \mathbb{R}^d_{li}$. Since condition~\ref{item:plusminus} holds for $A$, there is an open cone $\sigma_{\vect 0}$ such that $A^-_{\vect 0, \vect w} = \{\vect a\in A : \vect w\cdot \vect a < 0 \}$ is equal to $A^-_{\vect 0, \vect z}$ for any $\vect z\in \sigma_{\vect 0}\cap \mathbb{R}^d_{li}$. Since $A= A^-_{\vect 0, \vect w}$ by assumption, we have that $\vect a\cdot \vect z < 0$ for any $\vect a\in A$ and any $\vect z\in \sigma_{\vect 0}$. 
Set $V = \sigma\cap \sigma_{\vect 0}$.

Now, let ${\vect w'}\in V\cap \mathbb{R}^d_{li}$ and assume that there exists an infinite strictly decreasing sequence $T:=\{\vect w'\cdot (p^{-k_j}\vect a_j)\}_{j}$, where each ${\vect a_j}\in A$, and each $k_j>0$.  Then there must be infinitely many distinct integers $k_j$ in this sequence, or else some $Q_{i, \vect w'}$ would not be well ordered.  So, there is a strictly increasing subsequence $\{k_{n(j)}\}_{j}$ of the sequence $\{k_j\}_{j}$. Consider the associated subsequence of $T$:
\begin{equation}
\label{eqn:pjvsequence}
 \vect w' \cdot (p^{-k_{n(0)}}\vect a_{n(0)}) >  \vect w' \cdot(p^{- k_{n(1)}}\vect a_{n(1)}) > \cdots >  \vect w' \cdot(p^{-k_{n(j)}}\vect a_{n(j)}) > \cdots.
\end{equation}
Since the set $\{\vect w'\cdot\vect a_{n(j)}\}_{j}$ is well ordered, as $\vect w'\in V\cap \mathbb{R}^d_{li}$, it has a smallest element.  In particular, there must exist indices $s<t$ with $\vect w' \cdot\vect a_{n(s)} \leq \vect w' \cdot\vect a_{n(t)}$. Finally $\vect w' \cdot\vect a_{n(t)}$ is negative, as $\vect w'\in \sigma_{\vect 0}\cap \mathbb{R}^d_{li}$, and $k_{n(s)}< k_{n(t)}$, so we have 
\[
\vect w' \cdot(p^{-k_{n(s)}}\vect a_{n(s)})  \leq \vect w' \cdot(p^{-k_{n(s)}}\vect a_{n(t)}) < \vect w' \cdot(p^{-k_{n(t)}}\vect a_{n(t)}),
\]
which contradicts \eqref{eqn:pjvsequence}. Hence there is no infinite strictly decreasing sequence $T$, and the set $\{\vect w' \cdot\vect s: \vect s \in S\}$ is well ordered. 

\noindent \emph{Condition~\ref{item:cone}:} Since $A\in
\mathcal{A}_{\bf w}$, there is $N>0$, $\bfgamma \in \mathbb Q^d$, and
a pointed cone $C$ such that $\mathbf{w} \in
\intt(C^{\vee})$ and $A$ is contained in the intersection
  $$ (\bfgamma+C) \cap \left(\bigcup_{j \geq 0} \frac{1}{Np^j} \mathbb Z^d\right).$$
Since $\vect w\cdot \vect a<0$ for $\vect w \in A$, we must have that $\vect w \cdot \bfgamma<0$.
Let $C_1$ be the convex hull of $C$ and the ray spanned by $-\bfgamma$.
This is a pointed cone, and $\vect w \in C_1^{\vee}$
because $\vect w \in C^{\vee}$
and $-\vect w \cdot \bfgamma > 0$.
For $\vect a \in A$, since $\vect a - \bfgamma\in C$, we have that
$p^{-i}\vect a - p^{-i}\bfgamma\in C$, and that $(p^{-i}\vect a -
p^{-i} \bfgamma) + (1-p^{-i})(-\bfgamma) = p^{-i}\vect a - \bfgamma
\in C_1.$ Thus, $p^{-i}\vect a\in \bfgamma + C_1$, and we conclude
that $p^{-i}A$ is contained in
\begin{equation}\label{cond(b)TechnicalLemma}
  (\bfgamma+C_1) \cap\left(\bigcup_{j \geq 0} \frac{1}{Np^j} \mathbb Z^d\right).
  \end{equation}
Since $C_1$ does not depend on the particular $i>0$, we conclude that $S$ is contained in the set \eqref{cond(b)TechnicalLemma}.

\noindent \emph{Condition~\ref{item:convergence}:} Suppose that $\{p^{-k_j}\vect a_{j}\}_{j}$ is a sequence in $S$ with each $\vect{a}_j\in A$ such that $\{\vect w\cdot(p^{-k_j}\vect a_j)\}_{j}$ converges to an element of $\mathbb{R}$. 

If $\{p^{-k_j}\vect a_{j}\}_{j}$ is contained in $p^{-i}A$ for some fixed $i$ then we may apply condition \ref{item:convergence} for $A$ to conclude that the sequence converges to an element of $\mathbb{Q}^d$. 
If the sequence $\{p^{-k_j}\vect a_{j}\}_{j}$ is contained in a finite union of the sets $p^{-i}A$, then we may apply the condition \ref{item:convergence} part of the argument in item (iii) of the proof of Theorem \ref{thm:fieldfamily} to conclude that our original sequence $\{p^{-k_j}\vect a_{j}\}_{j}$ converges to an element of $\mathbb{Q}^d$. 

Assume that the sequence is not contained in a finite union of sets $p^{-i}A$. 
Then, there is a strictly increasing subsequence $\{k_{n(j)}\}_{j}$ of the sequence $\{k_j\}_{j}$. 
Since the elements of $\{\vect a_j\}_{j}$ are bounded in length by Lemma~\ref{lem:algClosedbounded}, the associated subsequence $\{p^{-k_{n(j)}}\vect a_{n(j)}\}_{j}$ of $\{p^{-k_j}\vect a_{j}\}_{j}$ converges to $\vect 0$. 
Therefore, since $\{\vect w\cdot(p^{-k_j}\vect a_j)\}_{j}$ converges by assumption, it must converge to $0$. 
We will show that this forces our original sequence $\{p^{-k_j}\vect a_{j}\}_{j}$ to converge to $\vect 0$.

Suppose otherwise. Then there is an $\epsilon>0$ and a subsequence $\{p^{-k_{m(j)}}\vect a_{m(j)}\}_{j}$ of $\{p^{-k_j}\vect a_{j}\}_{j}$  such that $p^{-k_{m(j)}}|\vect a_{m(j)}|>\epsilon$ for all $j\geq 0$. By Lemma~\ref{lem:algClosedbounded}, there exists an $M>0$ such that $|\vect a_{m(j)}|<M$ for each $j\geq 0$. Consequently, $p^{-k_{m(j)}}M>p^{-k_{m(j)}}|\vect a_{m(j)}|>\epsilon$, and we conclude that $p^{-k_{m(j)}}>\epsilon/M$ for all $j\geq 0$. 

Because $\{\vect w\cdot(p^{-k_{m(j)}}\vect a_{m(j)})\}_{j}$ converges to $0$, and the coefficients $p^{-k_{m(j)}}$ are  bounded below by $\epsilon/M$, we have that $\{\vect w\cdot\vect a_{m(j)}\}$ converges to $0$. 
As $A\in \mathcal{A}_{\vect w}$, it follows that $\{\vect a_{m(j)}\}_{j}$ converges to an element $\vect z \in \mathbb{Q}^d$ such that $\vect w\cdot \vect z = 0$. But the only $\vect{z}\in \mathbb{Q}^d$ such that $\vect w\cdot \vect z = 0$ is $\vect z = \vect 0$ since $\vect w\in \mathbb{R}^d_{li}$. This is a contradiction. Hence our original sequence $\{p^{-k_j}\vect a_{j}\}_{j}$ converges to $\vect 0$.

\noindent \emph{Condition~\ref{item:plusminus}:}
Since $A$ satisfies condition \ref{item:plusminus}, for each $\bfdelta \in \mathbb{Q}^d$, there is an open cone $\sigma_{\bfdelta}$ containing $\vect w$ such that for all $\vect w'\in \sigma_{\bfdelta}\cap \mathbb{R}^d_{li}$, we have $A^+_{\bfdelta, \vect w'} = A^+_{\bfdelta, \vect w}$ and $A^-_{\bfdelta, \vect w'} = A^-_{\bfdelta, \vect w}$. Fix $\bfgamma' \in \mathbb{Q}^d$. We consider three cases: (i) $\vect w\cdot\vect\bfgamma' > 0$, (ii) $\vect w\cdot \bfgamma' = 0$, and (iii) $\vect w \cdot \vect \bfgamma' < 0$. 

When $\vect w \cdot \vect \bfgamma' > 0$, we have $S = S^-_{\bfgamma',
  \vect w}$, and $S^+_{\vect w, \vect \bfgamma'} = \emptyset$, since
$\vect w\cdot \vect s < 0$ for all $\vect s\in S$.  Let $V :=\{\vect
v\in \mathbb{R}^d : \vect v \cdot \bfgamma' > 0\}$.  Then, for $\vect
a\in A$ and $\vect w' \in ({\sigma}_{\vect 0}\cap V) \cap
\mathbb{R}^d_{li}$, we have $\vect w'\cdot \vect a< 0$ since $A =
A^-_{\vect 0, \vect w} = A^-_{\vect 0, \vect w'}$. So, $\vect w'\cdot
\vect s< 0$ for any $\vect s\in S$. Since $\vect w'\in V$, we have
$\vect w'\cdot \vect s < \vect w'\cdot \vect \bfgamma' $ for every
$\vect s\in S$. Hence, $S = S^-_{\bfgamma',\vect w'}$ and
$S^+_{\bfgamma', \vect w'} = \emptyset$. We conclude that for each
${\vect w'}\in (\sigma_{\vect 0}\cap V)\cap \mathbb{R}^d_{li}$, we
have $S^+_{\bfgamma', \vect w'} = S^+_{\bfgamma',\vect w}$ and
$S^-_{\bfgamma', \vect w'} = S^-_{\bfgamma',\vect w}$.

Next consider the case when $\vect w\cdot \bfgamma' = 0$. Then
$\bfgamma' = \vect 0$ since $\vect w\in \mathbb{R}^d_{li}$. Then for
any $\vect w'\in \sigma_{\vect 0}\cap \mathbb{R}^d_{li}$, we have $S =
S^-_{\vect 0, \vect w} = S^-_{\vect 0, \vect w'}$ and $\emptyset =
S^+_{\vect 0, \vect w} = S^+_{\vect 0, \vect w'}$.

Finally, consider the case $\vect w\cdot \bfgamma' < 0$.  Let $i_0>0$
be such that for all $i\geq i_0$, we have $\vect w\cdot (p^{-i}\vect
a)> \vect w\cdot \vect \bfgamma'$, for all $\vect a\in A$. Such an
$i_0$ exists because $\vect w\cdot \bfgamma'<0$ and the set $\{\vect
w\cdot \vect a : \vect a\in A\}$ is well ordered.

Let $\vect w'\in \left( \bigcap_{i = 1}^{i_0} \sigma_{p^i\bfgamma'}
\right)\cap \mathbb{R}^d_{li}$.  Then,
\begin{enumerate}
\item for each $1\leq i \leq i_0$, we have $(p^{-i}A)^+_{\bfgamma', \vect w} = (p^{-i}A)^+_{\bfgamma', \vect w'}$ and $(p^{-i}A)^-_{\bfgamma', \vect w} = (p^{-i}A)^-_{\bfgamma', \vect w'}$, because $\vect w'\in \sigma_{p^i\bfgamma'}$;
\item by the choice of $i_0$ and since $\vect w'\in
  \sigma_{p^{i_0}\bfgamma'}$, we have that $p^{-i_0}A =
  (p^{-i_0}A)^+_{\bfgamma', \vect w} = (p^{-i_0}A)^+_{\bfgamma', \vect
    w'}$. 
  Furthermore, for each
      $i > i_0$ and each $\vect a\in A$, we have $\vect w'\cdot
      (p^{-i}\vect a)> \vect w'\cdot (p^{-i_0}\vect a)$. Thus, for
      each $i>i_0$, we have $(p^{-i}A)^+_{\bfgamma', \vect w'}= p^{-i}A$
and $(p^{-i}A)^-_{\bfgamma', \vect w'}=\emptyset$.
\end{enumerate}
Therefore, for all $i\geq 1$, we have $(p^{-i}A)^+_{\bfgamma', \vect w} = (p^{-i}A)^+_{\bfgamma', \vect w'}$ and $(p^{-i}A)^-_{\bfgamma', \vect w} = (p^{-i}A)^-_{\bfgamma', \vect w'}$. It follows that $S^+_{\bfgamma', \vect w} = S^+_{\bfgamma', \vect w'}$ and $S^-_{\bfgamma', \vect w} = S^-_{\bfgamma', \vect w'}$. 
\end{proof}

Let $\nu: \ourK \rightarrow \mathbb{R} \cup \{\infty\}$ be the valuation defined by 
\[
\nu(f):= \text{min}\{\vect w\cdot \vect a: \vect a\in \text{supp}(f)\}
\]
for $f\neq 0$ and $\nu(0) :=\infty$. 

\begin{rem}[Properties of $\ourK$ with the valuation $\nu$]\label{F:properties}
\begin{itemize}
\item[ ]
\item[(i)]\label{F:properties(i)} The valued field $\ourK$ has valuation ring ${R_{\vect w}} := \{f\in \ourK : \nu(f)\geq 0\}$. The valuation ring has maximal ideal $\mathfrak m= \{f\in \ourK : \nu(f) >0\}$. The residue field ${R_{\vect w}}/\mathfrak m$ is isomorphic to the field of coefficients $\kk$; the map which sends $\overline{f}\in {R_{\vect w}}/\mathfrak m$ to the constant term of $f\in {R_{\vect w}}$ gives the isomorphism. 
    \item[(ii)]\label{F:properties(ii)}  The value group of $\ourK$ is $\Q^d$, which is $n$-divisible for all $n$.  
       \item[(iii)]\label{F:properties(iii)} By \cite{Rayner1}*{Theorem 2}, $K_{\vect w}$ is a \emph{Henselian valued field}. This means that $R_{\vect w}$ satisfies Hensel's lemma: if $g(X)\in {R_{\vect w}}[X]$ and its reduction mod $\frak m$, $\overline{g}(X)\in \kk[X]$ has a simple root $a\in \kk$, then there exists a unique $b\in R_{\vect w}$ such that $g(b) = 0$ and $\overline{b} = a\in \kk$. 
\end{itemize} 
\end{rem}

We can now prove the main theorem of this section. Our proof is essentially the same as \cite[Theorem 5.3]{Saavedra} except that the third bullet point in Saavedra's proof is replaced by our Lemma~\ref{prop:algClosedpDiscrete}. 

\begin{theorem}
\label{thm:algebraicallyclosed}
Fix an algebraically closed field $\kk$ of characteristic $p>0$.  The
field $\ourK =\kk^{\Q^d}(\mathcal A_{\vect w})$ is algebraically
closed.
\end{theorem}

\begin{proof}
  We proceed by contradiction and assume that $\ourK$ is not algebraically closed. Then $\ourK$ admits a proper extension of finite degree.
  By \cite{Rayner1}*{Lemma 4} there exists $f\in \ourK$ with
  $\nu(f)<0$ such that the polynomial $X^p-X-f\in \ourK [X]$ is
  irreducible.  We express $f$ as the sum of two elements $f^+,f^-\in
  \ourK$ as follows:
\[
f^+(\vect v) =\begin{cases}
	f(\vect v) &  \vect v\in \supp(f), \vect w \cdot \vect v \geq 0 \\
	0 & \text{otherwise}	
\end{cases}
\quad
f^-(\vect v) = \begin{cases}
	f(\vect v) &  \vect v\in\supp(f), \vect w \cdot \vect v <0 \\
	0 & \text{otherwise.}
\end{cases}
\]
Since $\supp (f^+)$ and $\supp( f^-)$ are subsets of $\supp (f)\in \mathcal{A}_\vect w$, it follows that $\supp (f^+), \supp( f^-) \in \mathcal{A}_{\vect w}$. It suffices to prove the following two claims: 

 {\bf Claim 1.}  There exists $g\in \ourK$ that is a root of $X^p-X-f^+$.

{\bf Claim 2.}  There exists $h\in \ourK$ that is a root of $X^p-X-f^-$.

\noindent Indeed, for such $g$ and $h$ we have that $g+h$ is a root of $X^p-X-f$ which contradicts the irreducibility of $X^p-X-f$. Hence $\ourK$ is algebraically closed. 

\medskip

\noindent{\bf Proof of Claim 1.}
We have that $f^+$ belongs to the valuation ring ${R_{\vect w}}$ 
of the valued field $\ourK$, since $\nu(f^+)\geq 0$ by definition.  
We will use the Henselian property of ${R_{\vect w}}$ to find the root. Indeed, the reduction of $X^p-X-f^+$ modulo the maximal ideal $\frak{m}\subseteq {R_{\vect w}}$ is $X^p-X-f^+(\vect 0)\in \kk[X]$. 
Because $\kk$ is algebraically closed, $X^p-X-f^+(\vect 0)$ factors. Furthermore, $X^p-X-f^+(\vect 0)$ has simple roots because its derivative is $-1$.  By the Henselian property each of these roots lifts to a distinct root of $X^p-X-f^+$, so  the desired root $g$ exists.   This completes the proof of Claim 1.  
\medskip

\noindent{\bf Proof of Claim 2.} 
Define $h:\Q^d\to \kk$ by
\[
h(\vect v) = \begin{cases}
	\sum_{j=1}^{\infty}\left(f^-(p^j\vect v)\right)^{p^{-j}} & \vect w \cdot {\vect v} < 0 \\
	0 & \text{ otherwise}.
\end{cases}
\]
We will see in \eqref{eqn: righth} that this is exactly the root we want.

\item \emph{Subclaim a}.  $h$ is well defined.  

Since $\kk$ is
  algebraically closed of characteristic $p$, the $p^j$th roots exist
  and are unique. We now show that the infinite sum in the definition of $h$ is in
  fact a finite sum.  Assume for contradiction that there are infinitely many $p^j\vect v$s
  in $\supp(f^-)$.  Since $\vect w\cdot \vect v<0$, these infinitely many $p^j\vect v$s form an infinite strictly decreasing
  sequence. 
  This is a contradiction: $\supp(f^-)\in\mathcal A_{\vect w}$ and so $\supp(f^-)$ is well ordered by item (i) of the
  definition of field family (Definition \ref{dfn: fieldfamily}).
  
\item \emph{Subclaim b}. $\supp(h)\in \mathcal A_{\vect w}$. 

Since $\supp(h)$ is a subset of $S = \bigcup_{i=1}^\infty p^{-i} \supp(f) $, and 
$S\in \mathcal{A}_{\vect w}$, 
by Lemma~\ref{prop:algClosedpDiscrete},
 we conclude that $\supp (h)\in \mathcal{A}_{\vect w}$ as
$\mathcal{A}_{\vect w}$ is a field family.

\item \emph{Subclaim c}. $h^p-h-f^-=0$.

First suppose $\vect w \cdot \vect v <0$.  Then $\vect w \cdot p^{-1}\vect
v <0$.  We claim that
\[
h^p(\vect v) = h\left(p^{-1}\vect v \right)^p = \sum_{j=1}^{\infty}\left(f^-(p^{j-1}\vect v) \right)^{p^{-(j-1)}}.
\]
The first equality 
is the Frobenius homomorphism, and the second equality is by the definition of $h$. 
Hence, we have 
\begin{equation}
\label{eqn: righth}
\begin{aligned}
h^p(\vect v)-h(\vect v)-f^-(\vect v) &= \sum_{j=1}^{\infty}\left(f^-(p^{j-1}\vect v) \right)^{p^{-(j-1)}} - \sum_{j=1}^{\infty}\left(f^-(p^j\vect v)\right)^{p^{-j}} - f^-(\vect v) \\
	&= \sum_{j'=0}^{\infty}\left(f^-(p^{j'}\vect v)\right)^{p^{-j'}} - \sum_{j=1}^{\infty}\left(f^-(p^j\vect v)\right)^{p^{-j}} - f^-(\vect v) \\
	&= f^-(\vect v)-f^-(\vect v) = 0.
\end{aligned}
\end{equation}
On the other hand, when $\vect w \cdot \vect v \geq 0$
we have $h^p(\vect v)=h(\vect v)= 0=f^-(\vect v)$ by definition.  So
$h$
satisfies the polynomial $X^p-X-f^-$.
\end{proof}

\section{Toric Bertini Theorems}

\label{sec:toricBertini}

In this section we prove the main theorem of this paper: Theorem~\ref{t:maintoricbertini}.

If $\charr(\kk)>0$ set $p=\charr(\kk)$; otherwise set $p=1$.  For
$\mathbf{w} \in \R_{li}^d$, let $\ourK$ be the field of $p$-discrete
series in direction $\mathbf{w}$ with coefficients in $\kk$, with
variables $t_1,\dots,t_d$, as defined in Definition~\ref{dfn:K_w}.  If
$\charr(\kk)=0$, we set $p=1$ in this definition; this is still an
algebraically closed field by \cites{McD,ArocaIlardi}.  For an open cone $C$
containing $\mathbf{w}$, let $K_C$ be the subring of $\ourK$
consisting of those elements whose supports are well ordered with
respect to $\preceq_{\mathbf{w}'}$ for every $\mathbf{w}' \in C \cap
\R_{li}^d$, and for which the only allowable denominators in exponents
are powers of $p$.  The fact that $K_C$ is a subring follows from the
fact that unions and sums of well-ordered subsets of $\mathbb Q^d$ are
again well ordered.  We will use the following properties of $K_C$:
\begin{enumerate}
\item For $\alpha \in K_C$, there is $\bfgamma \in \mathbb Q^d$ with $\supp(\alpha) \subset \bfgamma+C^{\vee}$.  
\item Polynomials $\alpha \in \kk[t_1^{\pm 1},\dots,t_d^{\pm 1}]$
  are in $K_C$ for every cone $C$.
\item  If $C' \subseteq C$, then
  $K_{C}$ is a subring of $K_{C'}$.
  \end{enumerate}

Let $\genPuiseux$ be the field of generalized Puiseux series in one
variable $x$.  This consists of all formal power series with
coefficients in $\kk$ whose support is a well-ordered subset of
$\bigcup_{i=0}^\infty\frac{1}{Np^i}\Z$ for some integer $N$.

\begin{dfn}\label{dfn:substitution_map}
  Fix $\mathbf{w} \in \mathbb R^d_{li}$, and an open cone $C$
  containing $\mathbf{w}$.  For $\mathbf{n}= (n_1,\ldots,n_d) \in
  C \cap \mathbb Z^d$, and 
  $\bftheta=(\theta_1,\ldots,\theta_d)\in (\kk^\ast)^d$
  we define
$
\spec : K_C \rightarrow \genPuiseux
$
by 
$$\spec \left( \sum c_{\mathbf{u}} \mathbf{t}^{\mathbf{u}} \right) = \sum c_{\mathbf{u}}\bftheta^\mathbf{u} x^{\mathbf{n} \cdot \mathbf{u}}.$$
\end{dfn}

\begin{lemma}
  \label{lem:sub}
  For every $\mathbf{n}=(n_1,\ldots,n_d) \in C \cap \Z^d$, and every
  $\bftheta \in (\kk^*)^d$, the function $\spec$
is well defined and is a ring homomorphism from $K_C$ to $\kk\{\!\{x\}\!\}$.
\end{lemma}

\begin{proof}
We first observe that the condition on $K_C$ that the denominators of
exponents be only powers of $p$ ensures that the expression ${
  \bftheta}^{\vect u}$ is well defined, as $p$th roots are unique in
characteristic $p$.  Fix $\alpha \in K_C$. To finish showing that $\spec$ is well defined, we need to check that the
substitution does not map infinitely many terms of $\alpha$ to the
same term in $\genPuiseux$.  In other words, we need to check
that the fibers of the map
  \begin{equation} \label{eqtn:collapsemap} \supp(\alpha) \rightarrow \Q, \quad \mathbf{v} \mapsto \mathbf{n}
    \cdot \mathbf{v} \end{equation}
  are finite.  Suppose on the contrary that for some $r \in \Q$ the
  preimage $\{\mathbf{v} \in \supp(\alpha) : \mathbf{n} \cdot
  \mathbf{v} = r\}$ is infinite.  Consider an infinite non-repeating
  sequence in this set.  

  Since $C$ is an open cone and $\mathbf{n} \in C$, there are vectors
  $\mathbf{w}_1, \dots, \mathbf{w}_d \in C \cap \R_{li}^d$ such that
  $\mathbf{n} = a_1 \mathbf{w}_1+\cdots + a_d \mathbf{w}_d$, for some
  positive real numbers $a_1,\dots, a_d$.  Consider the infinite
  sequence above.  Since it is well ordered with respect to
  $\preceq_{\mathbf{w}_1}$, we can pass to an increasing subsequence.
We can now repeat this operation using the orders given by
  $\mathbf{w}_2, \dots, \mathbf{w}_d$, to get a sequence that is increasing with
  respect to all the orders given by $\mathbf{w}_1, \dots, \mathbf{w}_d$.
Thus the dot product of the sequence with $\mathbf{n}$
must be strictly increasing as well, contradicting the assumption that
it is constant.

This shows that the substitution is a well-defined map.  The fact that $\spec$ is a ring homomorphism follows from the fact that taking the dot product with $\mathbf{n}$ commutes with union and sum of subsets of $\mathbb Q^d$.
\end{proof}

Since $\kk[t_1,\dots,t_d] \subseteq K_C$ for all cones $C$, we can also define an analogous  homomorphism from $\kk[t_1,\dots,t_d,y]$:
$$\spec \colon \kk[t_1,\dots,t_d,y] \rightarrow \kk[x^{\pm 1},y]$$
by
$$\spec\left(\sum c_{\mathbf{u},j} {\vect t}^{\vect u}y^j\right) = \sum \spec(c_{\mathbf{u},j} {\vect t}^{\vect u}) y^j.$$

We say an element $\alpha \in K_C$ has unbounded support if the support of $\alpha$ is not contained in a
bounded region in $\R^d$.

\begin{lemma}
    \label{lem:unboundedsupport}
Let $\alpha \in K_C$ have unbounded support and be 
algebraic over $\kk(t_1,\dots,t_d)$.
Then there exists an open cone $C' \subseteq C$  containing $\mathbf{w}$
such that for any integer vector $\mathbf{n} \in C'$ and any $\bftheta  \in (\kk^*)^d$ the image
$\spec(\alpha)$
is not a polynomial in $\kk[x^{\pm 1}]$.
\end{lemma}

\begin{proof}
 Since $\alpha$ is algebraic over $\kk(t_1,\dots,t_d)$, there is a
 polynomial $h \in \kk[t_1,\dots,t_d,y]$ such that
 $h(t_1,\dots,t_d,\alpha) = 0$.  Since $\spec$ is a ring homomorphism
 for any $\mathbf{n} \in C \cap \mathbb Z^d$ and any $\bftheta \in (\kk^*)^d$, for such choices we 
 then have
   \[
   \spec(h)(x,\spec(\alpha)) = h ( \theta_1 x^{n_1}, \ldots, \theta_d x^{n_d}, \spec (\alpha)) = 0.
   \]
Suppose
 $\spec (\alpha)$
is a polynomial in $\kk[x^{\pm 1}]$.
In order to have
$\spec(h)(x,\spec(\alpha)) = 0$, there
must be two distinct monomials $x^{a_i} y^i$ and $x^{a_j} y^j$ in $\spec(h)$ that have
the same maximal $x$-degree after plugging in $\spec (\alpha)$ for $y$.  Then 
\[ a_i + i \cdot \deg_x\left(\spec (\alpha)\right) = a_j + j \cdot \deg_x \left(\spec (\alpha)\right)\]
so $i \neq j$ and we have $\deg_x \spec (\alpha) = \frac{a_i - a_j}{j-i}$.
Since each monomial in $\spec (h)$ is the specialization of a monomial in $h$, 
\[
\deg_x \spec(\alpha)
\leq \max\left\{\mathbf{n} \cdot \left(\frac{\mathbf{u}_r -
  \mathbf{u}_s}{s-r}\right) : {\bf t}^{\mathbf{u}_r} y^r \text{ and } {\bf t}^{\mathbf{u}_s} y^s \text{
    are monomials in }h\right\}.
  \]
For all directions $\mathbf{n}$ sufficiently close to $\mathbf{w}$,
the maximum is attained at the same pair of monomials ${\bf
  t}^{\mathbf{u}_i} y^i$ and ${\bf t}^{\mathbf{u}_j} y^j$ in $h$.
Explicitly, we may take $\mathbf{n}$ in the open cone $C_1$
containing $\mathbf{w}$ in the normal fan of the convex hull of the
$\frac{\mathbf{u}_r-\mathbf{u}_s}{s-r}$, which is also the normal cone
of $\frac{\mathbf{u}_i-\mathbf{u}_j}{j-i}$.

  Since $\alpha \in K_C$, there is $\bfgamma \in \mathbb Q^d$ with
  $\supp(\alpha) \subset \bfgamma+C^{\vee}$.  Choose an open cone
  $C_2 \subseteq C_1 \cap C$ whose closure is contained in the interior of
  $C$.  

Then $(\bfgamma+C^{\vee}) \setminus
(\frac{\mathbf{u}_i-\mathbf{u}_j}{j-i} + {C_2}^{\vee})$ is bounded,
but $\supp(\alpha)$ is unbounded,
so we can choose $\mathbf{v} \in \supp(\alpha) \cap \frac{\mathbf{u}_i-\mathbf{u}_j}{j-i} + {C_2}^{\vee}$.
Then since $\mathbf{v} \in \frac{\mathbf{u}_i-\mathbf{u}_j}{j-i} + {C_2}^{\vee}$, for all $\mathbf{n} \in
C_2$ we have $\mathbf{n} \cdot \mathbf{v} \geq \mathbf{n} \cdot
\frac{\mathbf{u}_i-\mathbf{u}_j}{j-i}$.
By axiom~\ref{item:plusminus} of $p$-discreteness for $\ourK$, there
is an open cone $C_3$ containing $\mathbf{w}$ for which for any
$\mathbf{n} \in C_3 \cap \mathbb Z^d$ the exponent $\mathbf{v}$ is
the only preimage of the map~\eqref{eqtn:collapsemap}, so does not get
cancelled after the substitution $\spec$.  Thus for $\mathbf{n} \in
C_2 \cap C_3$ the monomial $x^{\mathbf{n} \cdot \mathbf{v}}$ appears
in a term of $\spec(\alpha)$, contradicting the degree bound above.
Thus $\spec(\alpha)$ cannot be a polynomial for $\mathbf{n} \in C' :=
C_2 \cap C_3$.
\end{proof}

\begin{lemma}   \label{lem:unboundeddenominator}
Let $\alpha \in K_{C}$ be an element with infinite but bounded support.  Then there
exists an open cone $C' \subseteq C$ containing $\mathbf{w}$ and a sublattice $H \subset \Z^d$ such that for any integer vector
$\mathbf{n} \in (C' \setminus H) \cap \Z^d$, the substitution
$\spec(\alpha)$ is not a polynomial.  Moreover, the sublattice
$H$ can be chosen to have an arbitrarily large index in $\Z^d$.
\end{lemma}

\begin{proof}
The support of $\alpha$ must contain points in $\Q^d$ with coordinates
whose denominators (in reduced form) are arbitrarily large because
there are only finitely many elements in $\Q^d$ in a bounded region
with denominators smaller than a given bound.  Let $\mathbf{v}$ be an
element of $\supp(\alpha)$ with a non-integer coordinate.  By
axiom~\ref{item:plusminus} of $p$-discreteness, there is an open cone $C' \subseteq C$
containing $\mathbf{w}$ for which for any $\mathbf{n} \in C' \cap
\mathbb Z^d$ the term $\mathbf{v}$ is the only preimage of the
map~\eqref{eqtn:collapsemap}, so does not get cancelled by the
substitution $\spec$.
Let $H = \{\mathbf{n} \in \Z^d \mid \mathbf{n} \cdot \mathbf{v}
\in \Z\}$.   For all integer vectors $\mathbf{n} \in C'$
not in $H$, the map $\phi_{\mathbf{n}}^{\bftheta}$
sends the term $\mathbf{t^v}$
 in $\alpha$ to the term $\bftheta^{\mathbf{v}}x^{\mathbf{n} \cdot \mathbf{v}}$ in $\spec(\alpha)$  with
non-integer exponent, so $\spec(\alpha)$ is not a polynomial.
The lattice $H$ is a sublattice of $\Z^d$ but is not all of $\Z^d$ because
if the denominator (in reduced form) of $j$th coordinate of
$\mathbf{v}$ is $M > 1$, then $\mathbf{e}_j, 2\mathbf{e}_j, \dots, (M-1)\mathbf{e}_j$ are not
in $H$, so $H$ has index at least $M$ in $\Z^d$.
The element
$\mathbf{v}\in \supp(\alpha)$ can be chosen to have a coordinate with
an arbitrarily large denominator, so the sublattice $H$ to avoid can
be chosen to have an arbitrarily large index.
\end{proof}

Following Amoroso-Sombra \cite{amoroso2017factorization}, we say that for an irreducible variety $X$, a map $\pi : X
\rightarrow (\kk^*)^d$ has the {\em PB property } if for every isogeny (surjective group homomorphism  with finite kernel) $\lambda$ of $(\kk^*)^d$ 
the pullback $\lambda^* X := X
\times_{\lambda} (\kk^*)^d$ in \eqref{eqtn:pullback} is irreducible.
\begin{equation} \label{eqtn:pullback}
\begin{tikzcd}
 X
\times_{\lambda} (\kk^*)^d 
  \arrow[r] \arrow[d,"\lambda^*\pi"] & X \arrow[d,"\pi"] 
&   \\
 (\kk^*)^d \arrow[r, "\lambda"]
& (\kk^*)^d
\end{tikzcd}
\end{equation}

\begin{ex}
Consider the variety $X=\Var(x^2-yz^2)$.  The map $\pi:
X\rightarrow (\kk^*)^2$ projecting onto the first two coordinates does
not satisfy the PB property, since for the isogeny $\lambda:
(\kk^*)^2\rightarrow (\kk^*)^2$ given by  $(x,y)\mapsto (x,y^2)$ we have that
$X \times_{\lambda}
(\kk^*)^2=\Var(x^2-y^2z^2)=\Var(x-yz)\cup
\Var(x+yz)$ is reducible.
\end{ex}

\begin{theorem}\label{thm:main}
  Let $f \in \kk[t^{\pm 1}_1,\dots,t^{\pm 1}_d,y]$ be irreducible and  monic in $y$.
  Suppose that the projection of $\Var(f) \subseteq (\kk^*)^d \times \mathbb A^1$ 
  onto the first $d$ coordinates
  has the PB property.
  Fix $\mathbf{w} \in
  \mathbb R^d_{li}$.
  Then there exists an open cone $C$ containing
  $\mathbf{w}$ and finitely many sublattices $H_1,\dots, H_r \subset \Z^d$ such that for all
  vectors $\mathbf{n} \in (C \cap \Z^d)\setminus (H_1 \cup \cdots \cup
  H_r)$, and all $\bftheta \in (\kk^*)^d$ the Laurent polynomial
  $f(\theta_1 x^{n_1},\dots,\theta_d x^{n_d},y)
  \in \kk[x^{\pm
    1}, y]$ is irreducible.  Moreover, the sublattices can be
  chosen to be of arbitrarily high index.
\end{theorem}

\begin{proof}
Consider $f$ as a polynomial in $\ourK[y]$.  Since
$\ourK$ is algebraically closed, we can write
$$f = \prod_{i=1}^s (y-\alpha_i)$$
where $\alpha_i \in \ourK$.
By the construction of $\ourK$, there is $N>0$ for which every denominator of an
exponent appearing in some $\alpha_i$ can be put in the form $Np^j$
for some $j \geq 0$, where $p=1$ if $\charr(\kk)=0$.  The isogeny $\mu
\colon (\kk^*)^d \rightarrow (\kk^*)^d$ given by $t_i \mapsto t_i^N$
extends to an inclusion of fields $\mu \colon \ourK \rightarrow
\ourK$, and a map
$\ourK[y] \rightarrow \ourK[y]$, sending $y$ to $y$.
We then have $g:=\mu(f) = \prod_{i=1}^s (y-\mu(\alpha_i))$.  It
suffices to prove the theorem for $g$, as
$\spec(g) = \phi_{N\vect{n}}^{\bftheta}(f)$, and
if $\phi_{N\mathbf{n}}^{\bftheta}(f)$ is irreducible, so is
$\spec(f)$.  We thus henceforth assume that $N=1$.

Since $N=1$,
all denominators are powers of $p$, so by axioms~\ref{item:perturbation} and \ref{item:cone} of $p$-discreteness there is an open cone $C$ with $\alpha_i \in K_C$
for all $1 \leq i \leq s$.  Thus for all $\mathbf{n} \in C \cap \mathbb Z^d$ and all $\bftheta \in (\kk^*)^d$ we have
$$\phi_{\mathbf{n}}^{\bftheta}(f) = \prod_{i=1}^s (y - \phi_{\vect
  n}^{\bftheta}(\alpha_i)),$$ so all roots of $\phi_{\vect
  n}^{\bftheta}(f)$ have the form $\phi_{\vect n}^{\bftheta}(\alpha_i)$ for a root $\alpha_i$ of $f$.
Since $\spec(f)$ is monic, if it is not irreducible, then it can be factored into monic polynomials, and monic factors of $\spec(f)$ are images under $\spec$ of monic factors of $f \in \ourK[y]$.  Thus we can consider all of the finitely many ways to factor $f$ into
two monic polynomials in $\ourK[y]$ and show that ``most'' substitutions of
$\mathbf{n}$ do not make the factors into polynomials in $\kk[x^{\pm 1},y]$.
Since
$f$ has $\deg_y(f)$ roots in $\ourK$, there are $m :=
2^{\deg_y(f)-1}-1$ ways to factor $f$ into two monic polynomials.

Since the projection $\pi:\Var(f)\to (\kk^*)^d$ onto the first $d$
coordinates has the PB property, $f$ cannot be factored into
polynomials in $\ourK[y]$ whose coefficients in $\ourK$ have finite
support.  If it did have such a factorization, then the isogeny
$\lambda:t_i\mapsto t_i^{N'}$ would clear the common denominator $N'$ of
the exponents of the coefficients in the factorization, violating the
PB property.  If a factorization of $f$ in $\ourK[y]$ involves a
coefficient with unbounded support, then by
Lemma~\ref{lem:unboundedsupport} there is an open cone containing
$\vect w$ such that any integer vector $\mathbf{n}$ in this cone gives
a substitution that is not a polynomial factorization.  If a
factorization involves a coefficient whose support is infinite but
bounded, by Lemma~\ref{lem:unboundeddenominator} we can choose a
sublattice of index greater than $m$ for $\mathbf n$ to avoid.  Since
the union of $m$ lattices with index greater than $m$ cannot cover all
of $\Z^d$, the conclusion follows.
\end{proof}

\begin{rem}
  The use of the field $\ourK$, instead of one of the larger fields of
  generalized Puiseux series such as the one constructed in
  \cite{Saavedra}, was crucial for Theorem~\ref{thm:main}.  The fact
  that the polynomial $f$ factors completely is a consequence of the
  fact that $\ourK$ is algebraically closed
  (Theorem~\ref{thm:algebraicallyclosed}), which used
  axiom~\ref{item:plusminus} of Definition~\ref{dfn: p-discrete}.  The
  fact that the roots of the transformed polynomial $g$ all live in
  $K_C$ for some open cone $C$ is a consequence of
  axiom~\ref{item:perturbation} of Definition~\ref{dfn: p-discrete}.
  Without this reduction, the specialization homomorphism
  $\spec$ might not be well defined.
\end{rem}

We are now ready to prove the main theorem of this paper.  Recall that
a morphism $\psi \colon (\kk^*)^r \rightarrow (\kk^*)^d$ is given in
coordinates by $\psi(t_1,\dots,t_r) =
(c_1\mathbf{t}^{\mathbf{a}_1},\dots,c_d\mathbf{t}^{\mathbf{a}_d})$,
where $c_i \in \kk^*$ and $\mathbf{a}_i \in \mathbb Z^r$ for $1 \leq
i \leq d$, and $c_i=1$ for all $i$ if $\psi$ is an embedding of tori.
Let $A$ be the $r \times d$ matrix with columns the $\mathbf{a}_i$.
The morphism $\psi$ is an embedding if the matrix $A$ has rank $r$,
and the greatest common divisor of the $r \times r$ minors of $A$ is
one.  Changes of coordinates on $(\kk^*)^r$ correspond to row
operations on $A$, so we conclude that $r$-dimensional subtori of
$(\kk^*)^d$ correspond to rational points in the Grassmannian
$\Gr(r,d)$.

  \theoremstyle{theorem}
  \newtheorem*{t:maintoricbertini}{Theorem~\ref{t:maintoricbertini}}
  \begin{t:maintoricbertini}[Toric Bertini]
Let $\kk$ be an algebraically closed field of arbitrary characteristic.
Let $X$ be a $d$-dimensional irreducible subvariety of $(\kk^*)^n$
where $d \geq 2$, and let $\pi \colon (\kk^*)^n \rightarrow (\kk^*)^d$ be a
morphism with $\pi|_X$ dominant and finite.  Suppose that $\pi|_X$
satisfies PB.  Then for every $1 \leq r \leq d-1$ the set of
$r$-dimensional subtori $T \subseteq (\kk^*)^d$ with $\pi^{-1}(\bftheta
\cdot T) \cap X$ irreducible for all $\bftheta \in (\kk^*)^d$ is dense in
$\Gr(r,d)$.
\end{t:maintoricbertini}

\begin{proof}
    The morphism $\pi$ is given by $\pi(\mathbf{t})_j = c_j
    \mathbf{t}^{\mathbf{p}_j}$ for $1\leq j \leq d$, $c_j \in \kk^*$,
    and $\mathbf{p}_j \in \mathbb Z^n$.  Let $P$ be the $d \times n$
    matrix with rows $\mathbf{p}_1,\dots,\mathbf{p}_d$.  Integer row
    and column operations correspond to changes of coordinates on
    $(\kk^*)^d$ and $(\kk^*)^n$, so we may assume that $P$ is in Smith
    normal form, and thus the morphism $\pi$ is projection onto the
    first $d$ coordinates followed by an isogeny $\lambda \colon t_i
    \mapsto t_i^{d_i}$ and multiplication by
    $\mathbf{c}:=(c_1,\dots,c_d) \in (\kk^*)^d$.  Since $\pi|_X$
    satisfies PB, $X \times_{\lambda} (\kk^*)^d \subset (\kk^*)^n$ is
    irreducible, and the map $\lambda^*\pi \colon X \times_{\lambda}
    (\kk^*)^d \rightarrow (\kk^*)^d$ in \eqref{eqtn:pullback} is
    projection onto the first $d$ coordinates followed by
    multiplication by $\mathbf{c}$.
    Furthermore, observe that $\lambda^* \pi$ satisfies PB.  If $T
    \subset (\kk^*)^d$ is an $r$-dimensional subtorus with
    $(\lambda^*\pi)^{-1}(\bftheta \cdot T)$ irreducible for all
    $\bftheta \in (\kk^*)^d$, then $\pi|_X^{-1}(\lambda(\bftheta)
    \cdot T')$ is irreducible for $T' = \lambda(T)$.  It thus suffices
    to prove the theorem for $X \times_{\lambda} (\kk^*)^d$.  As the
    factor $\mathbf{c}$ can be absorbed into $\bftheta$, we may thus
    assume that $\pi$ is projection onto the first $d$ coordinates.

    We now reduce to the case that $X$ is a hypersurface.  
  Fix $a_1,\dots,a_n \in \kk$, and consider the morphism
  $\rho \colon (\kk^*)^n \rightarrow (\kk^*)^d \times \mathbb A^1$
  given by $(t_1,\dots,t_n) \mapsto (t_1,\dots,t_d,\sum_{i=1}^n a_i
  t_i )$.  Writing $\pi' \colon (\kk^*)^d \times \mathbb A^1 \rightarrow (\kk^*)^d$ for the projection onto the first $d$ coordinates, we have the following commuting diagram.
\begin{center}
\begin{minipage}[c]{1.5in}
\xymatrix{ X \subseteq (\kk^*)^n
  \ar[dr]^{\rho} \ar[dd]_{\pi} & \\ & (\kk^*)^{d} \times \mathbb A^1 \ar[dl]_{\pi'}
  \\ (\kk^*)^d & \\
}
\end{minipage}
\end{center}

For generic $(a_1,\dots,a_n)$ the morphism $\rho$ is birational (for
example, by following the proof of \cite{Hartshorne}*{Proposition
  I.4.9}; note that this is independent of the characteristic of the
field).  Let $U \subset X$ be an open set on which $\rho$ is an
isomorphism, and let $Z = X \setminus U$.  Since $\pi$ is a proper
morphism, as it is finite, we conclude that $\pi(Z)$ is a subvariety
of $(\kk^*)^{d}$ of dimension at most $d-1$.  Since $\rho$ is
birational, the variety $\overline{\rho(X)}$ is $d$-dimensional, so is
a hypersurface in $(\kk^*)^{d} \times \mathbb A^1$, defined by a
polynomial $f \in \kk[t_1^{\pm 1},\dots,t_d^{\pm 1},y]$.  As $\pi|_X$
is finite, the polynomial $\sum_{i=1}^n a_i t_i \in \kk[X]$ satisfies
a monic polynomial with coefficients in $\kk[t_1^{\pm
    1},\dots,t_d^{\pm 1}]$.  This monic polynomial must be a multiple
of~$f$, so $f$ is monic in $y$.

We next show that $\pi': \Var(f)\rightarrow (\kk^*)^d$ satisfies PB. Consider the Cartesian square
\begin{equation}\label{eq:PBsquare}
\begin{tikzcd} \Var(f)\times_\mu (\kk^*)^d \arrow[r, "\alpha"] \arrow[d, "\mu^*\pi'"] &\Var(f)\subseteq (\kk^*)^d\times \mathbb{A}^1 \arrow[d, "\pi'"] \\ (\kk^*)^d \arrow[r, "\mu"] & (\kk^*)^d \end{tikzcd}
\end{equation}
where $\mu$ is an isogeny.  Since $\Var(f)\times_\mu (\kk^*)^d$ is the
hypersurface defined by the specialization, determined by $\mu$, of the
irreducible polynomial $f$, all irreducible components of
$\Var(f)\times_\mu (\kk^*)^d$ are $d$-dimensional.  This, together
with the finiteness of $\alpha$ (which follows since $\mu$ is finite
and finite morphisms are stable under base change), implies that all
irreducible components of $\Var(f)\times_\mu (\kk^*)^d$ map by
$\alpha$ onto $\Var(f)$.

Since $\rho|_X$ is birational, there exist non-empty open sets
$U\subseteq X$ and $W\subseteq \Var(f)$ such that $\rho|_U:
U\rightarrow W$ is an isomorphism.  Because all irreducible components
of $V(f)\times_\mu (\kk^*)^d$ map by $\alpha$ onto the irreducible
$V(f)$, we have $\overline{\alpha^{-1}(W)} = \Var(f)\times_\mu
(\kk^*)^d$.  Thus, to see that $\Var(f)\times_\mu (\kk^*)^d$ is
irreducible, it suffices to show that $\alpha^{-1}(W) = W\times_\mu
(\kk^*)^d$ is irreducible.  Since the isomorphism $\rho|_U : U \rightarrow
W$ respects the maps to $(\kk^*)^d$, we have $W \times_{\mu} (\kk^*)^d
\cong U \times_{\mu} (\kk^*)^d$, so the irreducibility follows from
the fact that $\pi|_X$ satisfies $PB$.  Thus $\pi'$ satisfies PB.

We next observe that it suffices to prove the theorem for $\pi' \colon
\Var(f) \rightarrow (\kk^*)^d$.  Let $T_0$ be any given
$r$-dimensional subtorus of $(\kk^*)^d$.  Note that for every $\bftheta \in (\kk^*)^d$ every irreducible
component $Y$ of $\pi|_X^{-1}(\bftheta \cdot T_0)$ satisfies $\pi(Y) = \bftheta \cdot T_0$.
Indeed, since $\bftheta \cdot T_0$ is a complete intersection, by the
Principal Ideal Theorem (see for example, \cite{Eisenbud}*{Theorem
  10.2}) we have $\dim(Y) \geq \dim(\bftheta \cdot T_0)$.  Since $\pi$
is finite, $\pi(Y)$ is a closed set, and $\dim(\pi(Y)) = \dim(Y)$ (see
for example, \cite{Eisenbud}*{Proposition 9.2}), so we conclude that
$\pi(Y) = \bftheta \cdot T_0$.  It thus follows that if $T_0$ is a
subtorus of $(\kk^*)^{d}$ with $\bftheta \cdot T_0$ not contained in
$\pi(Z)$
and $\pi|_X^{-1}(\bftheta
\cdot T_0)$ reducible for some $\bftheta \in (\kk^*)^d$, then every
irreducible component of $\pi|_X^{-1}(\bftheta \cdot T_0)$ intersects
$U$, so ${\pi'}^{-1}(\bftheta \cdot T_0) \cap \Var(f) =
\overline{\rho(\pi|_X^{-1}(\bftheta \cdot T_0))}$ is reducible as
well.

To prove the theorem, it thus suffices to prove that we can choose an
$r$-dimensional subtorus~$T$ arbitrarily close to $T_0$ in $\Gr(r,d)$ with $\bftheta \cdot T \not
\subseteq \pi(Z)$, and ${\pi'}^{-1}(\bftheta \cdot T) \cap \Var(f)$
irreducible for all $\bftheta \in (\kk^*)^d$.

Note that if ${\pi'}^{-1}(\bftheta \cdot T_0) \cap \Var(f)$ is
reducible, then for any one-dimensional subtorus $\tilde{T} \subset
T_0$ the preimage ${\pi'}^{-1}(\bftheta \cdot \tilde{T}) \cap \Var(f)$
is also reducible if it is nonempty.  To see this, consider a
parameterization of the torus $T_0$ by coordinates $z_1,\dots,z_r$.
Specializing $f$ to this parameterization yields a reducible
polynomial $f' \in \kk[z_1^{\pm 1},\dots,z_r^{\pm 1},y]$.  The
pullback ${\pi'}^{-1}(\bftheta \cdot \tilde{T})$ is a further nonzero
specialization of $f'$, so remains reducible.

Let $A_0$ be an $r \times d$ matrix corresponding to $T_0$, and let
$\mathbf{w} \in \mathbb R^d_{li}$ be close to the first row of $A_0$
in the Euclidean topology on $\mathbb P^{d-1}_{\mathbb Q}$  By Theorem~\ref{thm:main}, there is an
open cone $C$ containing $\mathbf{w}$ and a finite collection
$\mathcal H$ of lattices of arbitrarily high index for which for any
$\mathbf{n} \in C$ not in any lattice in $\mathcal H$ the subtorus
$\tilde{T}=(t^{n_1},\dots,t^{n_d})$ has ${\pi'}^{-1}(\bftheta \cdot
\tilde{T}) \cap \Var(f)$ irreducible for all $\bftheta \in (\kk^*)^d$.
We assume that the lower bound on the index of the lattices has been
chosen to guarantee that $\mathbb Z^d \setminus \bigcup_{H \in
  \mathcal H} H$ is nonempty.  We have $\bftheta \cdot \tilde{T}
\subseteq \pi(Z)$ only if the tropicalization $\trop(\bftheta \cdot
\tilde{T})$, which is an affine line in $\mathbb R^d$ with direction
vector $\mathbf{n}$, is contained in $\trop(\pi(Z))$.  For generic
$\mathbf{n}$, the polyhedral complex $\trop(\pi(Z))$ does not contain
any lines in the direction $\mathbf{n}$.

 Choose $\mathbf{n} \in C \setminus \bigcup_{H \in \mathcal H} H$
 close to $\mathbf{w}$ in the Euclidean topology on $\mathbb P^{d-1}_{\mathbb Q}$,
 such that $\trop(\pi(Z))$ does not contain any lines in direction
 $\mathbf{n}$, and let $A$ be the matrix obtained by replacing the
 first row of $A_0$ by $\mathbf{n}$.  Let $T$ be the $r$-dimensional
 subtorus of $(\kk^*)^d$ corresponding to $A$.  This contains the
 one-dimensional subtorus $\tilde{T}=\{ (t^{n_1},\dots,t^{n_d}) : t
 \in \kk^* \}$, for which ${\pi'}^{-1}(\bftheta \cdot \tilde{T}) \cap
 \Var(f)$ is irreducible and $\bftheta \cdot \tilde{T} \not \subseteq
 \pi(Z)$ for all $\bftheta \in (\kk^*)^d$, so ${\pi'}^{-1}(\bftheta
 \cdot T) \cap \Var(f)$ is irreducible, and thus $\pi|_X^{-1}(\bftheta
 \cdot T)$ is irreducible for all such $\bftheta$.  As $T$ is close to
 $T_0$ in $\Gr(r,d)$, this proves the theorem.
\end{proof}

\begin{rem}  \label{r:densenessnecessary}
  Theorem~\ref{t:maintoricbertini} differs from the version in
  \cite{FMZ}, in that we only show that the set of exceptional tori are
  the complement of a dense set, rather than essentially finite.  As
  already discussed, the finiteness is not achievable without the PB
  condition we, and \cite{FMZ}, impose.  We do not know whether the PB
  condition is necessary for the conclusion of
  Theorem~\ref{t:maintoricbertini}.  In our proof the main use is to
  guarantee that the sublattices $H_1,\dots,H_r$ to avoid do not cover
  $\mathbb Z^d$, so the ``genericity'' condition for $\mathbf{n}$ is
  nonempty.  However we do not know an example of a projection $\pi$
  failing PB where the union of the lattices is all of $\mathbb Z^d$.
  In addition, it is possible that the stronger finiteness conclusion
  of \cite{FMZ} holds when we assume PB.  To prove this using our
  techniques, we would need a deeper understanding of the structure of
  the cones $C$ used in Theorem~\ref{thm:main}.  In characteristic
  zero, under mild hypotheses, McDonald \cite{McD} relates these cones
  to the normal fan of the fiber polytope of a certain projection of
  the Newton polytope of the polynomial $f$.  It would be interesting
  to understand to what extent that can be generalized to arbitrary
  characteristic.
\end{rem}

\section{Tropical Bertini Theorem}

\label{sec:tropicalBertiniTheorem}

One application of Theorem~\ref{t:maintoricbertini} is that the tropical Bertini theorem of \cite{maclagan2019higher} holds in arbitrary characteristic, and thus the $d$-connectivity of the tropicalization of $d$-dimensional irreducible varieties holds in arbitrary characteristic.

The \emph{tropicalization} of a variety $X \subseteq (\kk^*)^n$, where $\kk$ is a valued field with valuation $\val \colon \kk^* \rightarrow \mathbb R$,~is
\[\trop(X) = \cl( (\val(x_1),\dots,\val(x_n) ) : (x_1,\dots,x_n) \in
X(L) ),\] where $L/\kk$ is a nontrivially valued algebraically closed
field extension, and $\cl()$ is the closure in the usual Euclidean
topology on $\mathbb R^n$.  By the Structure Theorem (see, for
example, \cite{TropicalBook}*{Chapter 3}, or \cite{BieriGroves})
$\trop(X)$ is the support of a connected polyhedral complex.

In \cite{maclagan2019higher} Maclagan and Yu proved a tropical Bertini theorem in characteristic zero.  We now remove the characteristic assumption.

We identify rational affine hyperplanes in $\mathbb R^n$ with points
of $\mathbb P^n_{\mathbb Q}$.

\theoremstyle{theorem}
\newtheorem*{t:tropicalBertini}{Theorem~\ref{t:tropicalBertini}}
\begin{t:tropicalBertini}[Tropical Bertini]
Let $X \subset (\kk^*)^n$ be an irreducible $d$-dimensional variety,
with $d \geq 2$, over an algebraically closed field $\kk$ with $\mathbb Q$
contained in the value group.  The set of rational affine hyperplanes
$H$ in $\mathbb R^n$ for which the intersection $\trop(X) \cap H$ is
the tropicalization of an irreducible variety is dense in the
Euclidean topology on $\mathbb P^n_{\mathbb Q}$.
\end{t:tropicalBertini}

\begin{proof}
We only need to replace a few sentences from the proof of
\cite{maclagan2019higher}*{Theorem 5}.  Explicitly, the proof holds
essentially verbatim (with the forward reference to Theorem 8 replaced
by a forward reference to our Theorem~\ref{t:maintoricbertini}) until
the sentence of \cite{maclagan2019higher}*{Theorem 5} ``We can now
apply Theorem 8 to $\rho|_Y \colon Y \rightarrow (\kk^*)^d$'', where ``Theorem 8'' should again be replaced by Theorem~\ref{t:maintoricbertini} of this paper.  The proof then continues as follows.

By Theorem~\ref{t:maintoricbertini}, since $\rho|_Y$ satisfies PB, and
is dominant and finite, the set of $(d-1)$-dimensional subtori $T'
\subseteq (\kk^*)^d$ such that $\pi_X^{-1}(\bftheta \cdot T')$ is irreducible
for all $\bftheta \in(\kk^*)^d$ is dense in $\Gr(d-1,d) \cong \mathbb
P^{d-1}_{\mathbb Q}$.  By definition, this means that the set of
affine hyperplanes $\overline{H}=\trop(\bftheta \cdot  T')$ for such $\bftheta,T'$ is dense in
$\mathbb P^d_{\mathbb Q}$.

Fix one such $\overline{H} = \trop(\bftheta \cdot T')$, and let $H$ be the
hyperplane in $\mathbb R^n$ defined by $H = (\trop(\rho)^{-1}(\overline{H}))$.  Since $\rho$ is a monomial map, we have $H = \trop(\rho^{-1}(\bftheta \cdot T'))$.
As $\trop(\rho)$ is injective on every maximal face of $\trop(Y)$, the intersection $H \cap \trop(Y)$ is transverse, so by the Transverse Intersection Lemma~\cite{OssermanPayne}*{Theorem~1.1}, \cite{BJSST}*{Lemma 15},
\cite{TropicalBook}*{Theorem 3.4.12},
$$H \cap \trop(Y) = \trop(\rho^{-1}|_Y(\bftheta \cdot T')).$$ Thus $H
\cap \trop(Y)$ is the tropicalization of the irreducible variety
$\rho|_Y^{-1}(\bftheta \cdot T')$.  Note that 
\[
\trop(\mu'(\rho|_Y^{-1}(\bftheta \cdot T'))) = \trop(\mu')(H \cap \trop(Y)) = \trop(\mu')(H) \cap \trop(X) = \trop(\pi)^{-1}(\trop(\mu)(\overline{H})).
\]  
Since $\trop(\mu) \in \GL(n,\mathbb Q)$, and the set of $\overline{H}$ is dense in $\mathbb P^d_{\mathbb Q}$, the set of $\trop(\mu)(\overline{H})$ is dense in $\mathbb P^d_{\mathbb Q}$ as required.
\end{proof}

In \cite{maclagan2019higher} Theorem~\ref{t:tropicalBertini} is used to prove
a higher connectivity theorem for tropical varieties.  While the
statement of \cite{maclagan2019higher}*{Theorem 1} assumes that the field $\kk$
has characteristic zero, as observed in \cite{maclagan2019higher}*{Remark 11}
this is only needed to apply the tropical Bertini theorem, so in light
of Theorem~\ref{t:tropicalBertini} we have the following
generalization of \cite{maclagan2019higher}*{Theorem 1}.  A pure polyhedral
complex $\Sigma$ of dimension $d$ is $d$-connected through codimension
one if it is still connected after removing any $d-1$ closed facets.

\begin{theorem}
\label{thm:connectedness}
Let $\kk$ be a field that is either algebraically closed, complete, or
real closed with convex valuation ring.  Let $X$ be a $d$-dimensional
irreducible subvariety of $(\kk^*)^n$. 
Let $\Sigma$ be a pure
$d$-dimensional rational polyhedral complex with support $|\Sigma| =
\trop(X)$.  Write $\ell$ for the dimension of the lineality space of
$\Sigma$.  Then $\Sigma$ is $(d-\ell)$-connected through codimension
one.
\end{theorem}

\medskip 
\noindent
\footnotesize {\bf Authors' addresses:}
\smallskip 

%\vspace{-1mm}
\noindent{Department of Mathematics, Kalamazoo College,  MI,  USA} \hfill {\tt francesca.gandini@kzoo.edu}

\noindent{School of Mathematics, 
The University of Edinburgh,
Edinburgh EH9 3JZ,
United Kingdom} \hfill {\tt m.hering@ed.ac.uk}

\noindent{Mathematics Institute, University of Warwick, Coventry CV4 7AL, United Kingdom} \hfill {\tt D.Maclagan@warwick.ac.uk}

\noindent{Department of Mathematics: Algebra and Geometry, Ghent University, Belgium 
  \hfill {\tt fatemeh.mohammadi@ugent.be}\\
 Department of Mathematics and Statistics, UiT - 
  The Arctic University of Norway, Troms\o, Norway
} 

\noindent{Department of Mathematics and Statistics, McMaster University, Canada} \hfill {\tt rajchgot@math.mcmaster.ca}

\noindent{School of Mathematics, Georgia Institute of Technology, Atlanta GA, USA} \hfill {\tt wheeler@math.gatech.edu}

\noindent{School of Mathematics, Georgia Institute of Technology, Atlanta GA, USA} \hfill {\tt jyu@math.gatech.edu}

\end{document}